\documentclass[12pt, twoside]{article}
\usepackage{amsmath,amsthm,amssymb}
\usepackage{times}
\usepackage{enumerate}
\usepackage{color}

 \def\sqr#1#2{{\,\vcenter{\vbox{\hrule height.#2pt\hbox{\vrule width.#2pt
height#1pt \kern#1pt\vrule width.#2pt}\hrule height.#2pt}}\,}}
\def\bo{\sqr44\,}
\def\q{\quad}

\def\titlerunning#1{\gdef\titrun{#1}}
\makeatletter
\def\author#1{\gdef\autrun{\def\and{\unskip, }#1}\gdef\@author{#1}}
\def\address#1{{\def\and{\\\hspace*{18pt}}\renewcommand{\thefootnote}{}%
\footnote {#1}}%
\markboth{\titrun}{\titrun}}
\makeatother

\newtheorem{theorem}{Theorem}[section]
\newtheorem{corollary}[theorem]{Corollary}
\newtheorem{lemma}[theorem]{Lemma}
\newtheorem{proposition}[theorem]{Proposition}

\theoremstyle{definition}
\newtheorem{definition}[theorem]{Definition}
\newtheorem{remark}[theorem]{Remark}
\newtheorem{example}[theorem]{Example}

\newtheorem*{xrem}{Remark}

\newtheorem*{xthm}{Theorem}
\newtheorem*{xmainthm}{Main Theorem}
\newtheorem*{xdefn}{Definition}
\newtheorem*{xexam}{Example}

\numberwithin{equation}{section}

\begin{document}

\baselineskip=16pt


\titlerunning{C-H. Chu}

\title{ A Denjoy-Wolff theorem for bounded \\ symmetric domains}

\author{ Cho-Ho Chu \\
\it\small School of Mathematical Sciences, Queen Mary, University of London, London E1 4NS, UK}

\date{}

\maketitle

\address{{\it E-mail address:} c.chu@qmul.ac.uk}
\address{\footnotesize This research was partly supported by EPSRC, UK (grant no. EP/R044228/1).}
\address{To appear in J. Functional Analysis}

\vspace{-.3in}

\begin{abstract}
   Let $D$ be a  bounded symmetric domain of finite rank, realised as the open unit ball of  a complex Banach space,
which can be infinite dimensional.
Given a fixed-point free compact holomorphic map $f: D\longrightarrow D$, with iterates
$$f^n = \underbrace{f \circ \cdots \circ f}_{\mbox{$n$-times}},$$
such that the limit points of one orbit
$\{a, f(a), f^2(a), \ldots\}$
in an $f$-invariant horoball of unit hororadius lie in the extended Shilov boundary of $D$, we show that there is
a holomorphic boundary component $\Gamma$ of $D$, with closure $\overline \Gamma$,
such that $\ell(D) \subset \overline \Gamma$ for {\it all} subsequential limits $\ell = \lim_{k\rightarrow \infty} f^{n_k}$ of  $(f^n)$.

This generalises the Denjoy-Wolff theorem for a fixed-point
free holomorphic self-map $f$ on the disc $\mathbb{D}=\{z\in \mathbb{C}: |z|<1\}$.\\
\end{abstract}

\vspace{-.2in}


{\footnotesize {\it Keywords:} Bounded symmetric domain. Holomorphic map. Iteration.  Denjoy-Wolff theorem.
Horoball. Horofunction. JB*-triple.}

 \section{Introduction}

The celebrated Denjoy-Wolff theorem \cite{d,w}  for the complex open unit disc
$$\mathbb{D}=\{z\in \mathbb{C}: |z|<1\}$$
 asserts that the iterates $$f^n=\underbrace{f \circ \cdots \circ f}_{\mbox{$n$-times}}$$ of a fixed-point free
holomorphic map $f: \mathbb{D}\longrightarrow \mathbb{D}$ always converge locally uniformly
to a constant function $f_0(\cdot)=\xi$ with $|\xi|=1$.

The unit disc $\mathbb{D}$ is
the unique one-dimensional bounded symmetric domain of  rank one (up to biholomorphism).
The Denjoy-Wolff  theorem has been extended to higher dimensional
rank-$1$ symmetric domains, namely, the complex Euclidean balls, by Herv\'e \cite{h}, but fails for higher rank bounded symmetric domains,
even for the rank-$2$  bidisc $\mathbb{D}^2= \mathbb{D}\times \mathbb{D}$ \cite{h1}.

There have been various  generalisations of the Denjoy-Wolff theorem to higher and infinite dimensional bounded
domains (e.g. \cite{ab,lie, chulevico,rs}).
 Notably, a definitive generalisation to the bidisc $\mathbb{D}^2$ has been achieved by Herv\'e \cite{h1},
which can be formulated by stating that the images $\ell(\mathbb{D}^2)$ of all subsequential limits $\ell = \lim_k f^{n_k}$,
in the topology of locally uniform convergence,
of the iterates $(f^n)$  of a fixed-point free holomorphic self-map $f$ on $\mathbb{D}^2$ are contained in the
closure $\overline \Gamma$ of one single holomorphic boundary component $\Gamma$ of $\mathbb{D}^2$.

The holomorphic boundary components of $\mathbb{D}^2$ are sets of the form
$$\{(\xi,\eta)\}, \quad \{\xi\} \times {\mathbb{D}}, \quad {\mathbb{D}}\times \{\eta\}\qquad(|\xi|=|\eta|=1)$$
which are faces of $\mathbb{D}^2$.
Hence three distinct possibilities for the iterates $(f^n)$ on $\mathbb{D}^2$ can occur, as shown in \cite{h1}.
The holomorphic boundary components of $\mathbb{D}$ are the singletons $\{\xi\}$ with $|\xi|=1$. In view of this,
Herv\'e's result  generalises completely the Denjoy-Wolff theorem to the bidisc.

Despite many forms of generlisations of the Denjoy-Wolff theorem
to various higher dimensional domains delivered by many authors, the question
of extending Herv\'e's version to all (finite dimensional) bounded symmetric domains has remained illusive
since Herv\'e's publication \cite{h1} in 1954.
 Herv\'e's approach is specific to the bidisc $\mathbb{D}^2$ and in the case of polydiscs $\mathbb{D}^d$ ($d>2)$,
 a version of the Denjoy-Wolff theorem
 has been obtained in \cite{ab}, where it is shown that the images of all subsequential limits $\lim_k f^{n_k}$ for a fixed-point free
 holomorphic self-map $f$ are contained in a closed set larger than a union $T$ of holomorphic boundary components
 of the polydisc.  A similar result has also been shown in \cite{cr} for a finite Cartesian product of open unit balls of Hilbert spaces.
Nevertheless, it was  remarked in \cite{ab}  that {\it ``it is impossible to determine a priori''} that
 the image of a subsequential limit is contained in a particular
  component of $T$, and for this, {\it ``it is necessary to know something more about the map $f$''} .

To address the last remark, we first observe that the behaviour of the subsequential limits $\lim_k f^{n_k}$ of
a fixed-point free compact holomorphic self-map $f$ on a domain
$D$ are inevitably affected by  the geometry of the topological boundary $\partial D$ of $D$
since $\partial D$ {\it attracts} all the orbit  limit points, that is to say, the limit points  of each and every
orbit
$$ \mathcal{O}(a)= \{a, f(a), f^2(a), \ldots\} \qquad ( a\in {D})$$
 lie in $\partial {D}$ (cf.\,Lemma \ref{kry}).

In the case of Euclidean (unit) balls $\mathbb{B}_d \subset \mathbb{C}^d$ $(d\in \mathbb{N})$,
 a special  geometric feature is that  these domains are {\it strictly convex},
that is, $\partial \mathbb{B}_d$
 {\it coincides} with its (Bergman) Shilov boundary $\Sigma(\mathbb{B}_d)$,
consisting of  extreme points of the closure $\overline{\mathbb{B}}_d$,
and consequently all orbit limit points accumulate in $\Sigma(\mathbb{B}_d)$.
This fact is crucial for the Denjoy-Wolff theorem for $\mathbb{B}_d$ (cf. proof of Theorem \ref{single}).
 The Euclidean balls (and the open unit balls of Hilbert spaces)
are the only bounded symmetric domains which are strictly convex and therefore,  for other domains $D$
with Shilov boundary $\Sigma(D)$, the orbit limit points may
just lie in $\partial {D}\backslash \Sigma(D)$. In other words, there is no guarantee that
$\Sigma(D)$ would contain any orbit limit point.
In view of this, one would seek
 a larger set in $\partial D$ to attract orbit limit points, playing the role of $\Sigma(\mathbb{B}_d)$ in Euclidean balls.

Indeed, there is  a natural candidate sandwiched between $\Sigma(D)$ and $\partial D$, namely, the
real analytic manifold $T(D)$ consisting of the limits
of geodesic rays in $D$ emanating from a base point. These limits  are called {\it  tripotents}, which will be explained
in more detail in the next section. In the case of rank-$1$ domains $D$, including the open unit balls of Hilbert spaces, we actually have
$$\Sigma(D) = T(D)=\partial D$$ but for other domains
the inclusions $ \Sigma(D) \subset T(D) \subset \partial D$ are proper. On the other hand, if $D$ is a polydisc,
$T(D)$ coincides with its submanifold $T_1(D)$ of {\it structural} tripotents (also defined in the next section). If $D$ is a Lie ball,
then we have $\Sigma(D)=T_1(D)\neq T(D)$. In view of these examples,
one is led to consider the  submanifold $\Sigma^*(D) = \Sigma(D) \cup T_1(D)$ of $T(D)$  so that
$\Sigma^*(D)=T(D)$ for both polydiscs and rank-$1$ domains $D$. We call $\Sigma^*(D)$ the
{\it extended Shilov boundary} of $ D$.
 Although $\Sigma^*(D)$ is a larger manifold than $\Sigma(D)$,
it is still inconclusive that $\Sigma^*(D)$  would contain any orbit limit point.
This appears to be the crux of the matter and explains the obstacle to extending fully the Denjoy-Wolff theorem
 to other domains. If we remove this obstacle (by letting $\Sigma^*(D)$ attract enough orbit limit points),
 then a full generalisation of the Denjoy-Wolff theorem
emerges.

 In this paper, we prove  in Theorem \ref{dw} the following extension of
 the Denjoy-Wolff theorem
 to all finite-rank bounded symmetric domains,
 which include all finite dimensional ones, but can also be infinite dimensional. \\

\begin{xmainthm}
{\it  Let $D$ be a bounded symmetric domain of finite rank, realised as the open unit ball of a complex Banach space, and $f: D
  \longrightarrow D$  a  fixed-point free compact holomorphic map. If the extended Shilov boundary $\Sigma^*(D)$
  contains the limit points of
one orbit $\mathcal{O}(a)$ in an $f$-invariant horoball of unit hororadius,   then
there is a holomorphic boundary component $\Gamma$ in the boundary of $ D$ such that
$\ell(D) \subset \overline \Gamma$ for all subsequential limits $\ell= \lim_k f^{n_k}$ of $(f^n)$,
in the topology of locally uniform convergence.

If $D$ is the bidisc or the open unit ball of a Hilbert space, the condition on $\Sigma^*(D)$ is superfluous.}
\end{xmainthm}

In the case of the open unit ball $D$ of a Hilbert  space (including $\mathbb{D}$), the assumption on $\Sigma^*({D})$ is vacuous since
 $\Sigma^*({D})$ is the boundary of $D$ as noted earlier. In the exceptional case of the bidisc $\mathbb{D}^2$,
we have already noted that Herv\'e has shown the above result without the condition on $\Sigma^*(\mathbb{D}^2)=T(\mathbb{D}^2)$.
However, Herv\'e's proof is protracted and only applicable to $\mathbb{D}^2$. For completeness, we provide
 in the Appendix a simplification of Herv\'e's proof, which exposes the special role of  $\mathbb{D}^2$.

The compactness assumption on $f$ in the theorem is only necessary if $D$ is contained in
an infinite dimensional Banach space $V$
since the result is false without this condition even in the case of a rank-$1$ infinite dimensional
domain, namely, the open unit ball of an infinite dimensional complex Hilbert space \cite{cm}.
We say that $f: D \longrightarrow D \subset V$ is {\it compact} if the closure $\overline{f(D)}$ is compact
in $V$.
If $V$ is finite dimensional, then a continuous map $f : D \longrightarrow D$ is always compact.

For {\it non-compact} holomorphic self-maps $f$ on an infinite dimensional rank-$1$ domain $D$, of which the
open unit balls of infinite dimensional Hilbert spaces are the only example,
we show in Theorem \ref{single} the two alternatives for all subsequential limits  $\ell= \lim_k f^{n_k}$  to conclude the paper.

By a {\it horoball} we mean a (nonempty) convex subdomain in $D$ of the form
$$\{x\in D: F(x) <s\} \qquad (s>0)$$
for some continuous function $F: D \longrightarrow (0,\infty)$, where the defining function $F$  is the exponential of a {\it horofunction}
on $D$ (\`a la Gromov \cite{grom}) and $s$ is called the {\it hororadius}. It is called {\it $f$-invariant} if
$$f(\{x\in D: F(x) <s\}) \subset \{x\in D: F(x) <s\}.$$
The $f$-invariance of all horoballs of hororadius $s> 0$ is equivalent to $F\circ f \leq F$
in which case, $F$ is called $f$-invariant.
It has been shown
in \cite[Theorem 4.4]{horo} that  $D$ contains an abundance of $f$-invariant horoballs
and is actually entirely covered by them, given any
 fixed-point free compact holomorphic self-map $f$ on $D$. Relevant details of horoballs
and horofunctions will be given in Section \ref{Horo}.

The main ingredients of our novel approach, as hinted above,  are the horoballs,  defined in \cite{horo} and (\ref{ball}),
 and their defining functions, which are related to Gromov's notion of horofunctions in \cite{ball,grom},
together with the Jordan algebraic structures of  symmetric domains, in particular,
the Bergman operator introduced in (\ref{b1}), which plays a prominent role
 and facilitates several computations involving the horoballs and horofunctions.

Given a  fixed-point free compact holomorphic map
$f: D \longrightarrow D$, a key in proving Theorem \ref{dw} is the existence of
$f$-invariant horoballs in $D$,  in other words, the existence of an $f$-invariant function $F$,
which  is exposed in Lemma \ref{g}. The function $F$
 is defined by a sequence in $D$ converging to the
boundary. A crucial element in our arguments involves the computation of $F$ with such a sequence.
  Theorem \ref{cr1} provides another key
 which reveals that the intersection of all  closed horoballs defined by $F$
is exactly the closure of a holomorphic boundary component of $D$. To accomplish all this, we rely on
a detailed analysis of the
Jordan algebraic structures of a symmetric domain to provide an explicit and tractable description of the horoballs and the
 associated horofunctions, in terms of the Bergman operators.

We begin by introducing Jordan structures in the next section and prove several relevant results
for later applications. We discuss holomorphic boundary components in the following section.
After this, we derive a number of important properties of  horoballs and their defining functions, which
 are used to prove the Denjoy-Wolff  theorem in the final section.
 Various novel results we prove for Jordan structures and horofunctions on bounded symmetric domains
 may be of  independent interest.

In what follows, all Banach spaces are over the field of complex numbers, unless stated otherwise,
and the complex vector space $H(D, W)$ of holomorphic maps
from a {\it domain}, that is, an open connected set, $D$ in a Banach space
to some Banach space $W$ is equipped with the topology of {\it  locally uniform
convergence} so that a sequence $(f_n)$ in  $H(D, W)$ converges to some $f \in H(D,W)$ in this topology
if and only if it converges to $f$
uniformly on any finite union of closed balls strictly contained in $D$ (cf.\,\cite[1.1]{book2}).

The topological boundary of a bounded domain $D$ in a Banach space will always be denoted by
$\partial D$.

\section{Jordan structures of bounded symmetric domains}

A bounded domain $D$ in a complex Banach space is called {\it symmetric} if at each point $p\in D$,
there is a (necessarily unique) symmetry $s_p : D \longrightarrow D$, which is a biholomorphic map
such that $s_p \circ s_p$ is the identity map and $p$ is the only fixed-point of $s_p$ in a neighbourhood of $p$.
Bounded symmetric domains are examples of symmetric Banach manifolds, which have been
discussed in detail in \cite{u} (see also \cite{book2}).
In finite dimensions,  they are exactly  the class of Hermitian symmetric spaces
of non-compact type via  Harish-Chandra's realisation \cite{har}, and have been classified by \'E. Cartan using Lie theory
\cite{cartan}.

The key to applying Jordan theory in our setting is the seminal result of Kaup \cite{k}, and its finite dimensional
precursor by Loos \cite{LOOS}, asserting that
bounded symmetric domains are exactly (via biholomorphisms) the open unit balls of the complex Banach spaces
equipped with a particular Jordan structure, called {\it JB*-triples}. This offers a Jordan algebraic description as well as
a unified treatment of bounded symmetric domains in both  finite and infinite dimensions. Nevertheless, this Jordan
structure is closely related to the Lie structure associated to a bounded symmetric domain $D$. Indeed, the JB*-triple
structure is derived from the Lie triple system $\frak p$ in the Cartan decomposition
\begin{equation}\label{rank}
\frak g = \frak k \oplus \frak p
\end{equation}
of the Lie algebra $\frak g $ of complete holomorphic vector fields on $D$, induced by the symmetry $s_p$ at a point
$p\in D$.

A {\it JB*-triple} is a complex Banach space $V$ equipped with
 a continuous triple product
$$\{\cdot,\cdot,\cdot\} : V\times V \times V \longrightarrow V,$$
called a {\it Jordan triple product}, which is symmetric and linear in the outer variables, but conjugate linear in the
middle variable, and satisfies the {\it triple identity}
$$\{a,b,\{x,y,z\}\} = \{\{a,b,x\},y,z\}- \{x, \{b,a,y\}, z\} + \{x,y,\{a,b,z\}\} \quad (a,b,x,y,z \in V)$$
such that the bounded linear map
$$a\bo a: x\in V\mapsto  \{a,a,x\} \in V \qquad (a, x\in V)$$
satisfies
\begin{enumerate}
\item[(i)] $a\bo a$ is hermitian, that is,
$\|\exp it(a\bo a)\|=1$ for all $t \in \mathbb{R}$;
\item[(ii)]
$a\bo a$ has a non-negative spectrum;
\item[(iii)] $\|a\bo a\|=\|a\|^2 .$
\end{enumerate}

From the triple identity, one can derive the  useful identity
\begin{equation}\label{jp1}
\{\{x,y,x\}, z, \{x,y,x\}\} = \{x,y,\{x,z,x\},y,x\} \qquad (x,y,z \in V)
\end{equation}
in a JB*-triple $V$  \cite[Lemma 1.2.4]{book}.

\begin{xrem}  We note from  Kaup's result that a complex Banach space is a JB*-triple if and only if its open unit ball
is a symmetric domain  (cf. \cite[Theorem 2.5.27]{book}) and further, a bounded symmetric domain $D$ is biholomorphic
to the open unit ball of a {\it unique} JB*-triple $V$, in the sense that if $D$ is biholomorphic to the open
unit ball of another JB*-triple $W$, then $V$ and $W$ have identical JB*-triple structure via an isomorphism
(cf.\,\cite[p.\,176]{book2}). Hence in the sequel, we can, and will,
always identify a bounded symmetric domain with the open unit ball of a JB*-triple
unambiguously, and define its {\it dimension}, $\dim D$, to be that of the underlying JB*-triple.
\end{xrem}

Examples of JB*-triples include Hilbert spaces and C*-algebras. The Jordan
triple product in a Hilbert space $V$ with inner product $\langle \cdot,\cdot\rangle$  is given by
$$\{a,b,c\} = \frac{1}{2}(\langle a,b \rangle c +\langle c, b\rangle a) \qquad (a,b,c \in V).$$
In particular, $(\mathbb{C}, |\cdot|)$ is a JB*-triple with the complex modulus $|\cdot|$
and the standard inner product. As usual,  the complex conjugate of a complex number $\alpha$ is denoted
throughout by $\overline \alpha$.

In a C*-algebra $\mathcal{A}$, the Jordan triple product is
$$\{a,b,c\} = \frac{1}{2}(ab^*c + cb^*a) \qquad (a,b,c \in \mathcal{A}).$$
More generally, the Banach space $L(H,K)$ of bounded linear operators between Hilbert spaces $H$ and $K$
 is a JB*-triple in
the triple product $\{a,b,c\} = \frac{1}{2}(ab^*c + cb^*a)$, where $b^*$ denotes the adjoint operator of $b$.
The symbol $L(K,K)$ is shortened to $L(K)$.

\begin{xdefn} The open unit ball of a complex Hilbert space will be called
a {\it Hilbert ball}.
\end{xdefn}

\begin{example}\label{ellinfty}
 Let $\{V_\alpha\}_{\alpha \in \Lambda}$ be a family of JB*-triples.  We define their $\ell_\infty$-sum
$V$ to be the direct sum
$$V=\bigoplus_{\alpha \in \Lambda} {}_{_\infty} V_\alpha = \{ (v_\alpha) \in \mbox{\Large$\boldsymbol\times$}_\alpha V_\alpha
: \sup_\alpha \|v_\alpha\| < \infty\},$$
equipped with the $\ell_\infty$-norm
$$\|(v_\alpha)\|_\infty : =\sup_\alpha \|v_\alpha\|.$$
Then $V$ is a JB*-triple in the coordinatewise triple product
$$\{(a_\alpha), (b_\alpha), (c_\alpha)\} = ( \{a_\alpha, b_\alpha, c_\alpha\}).$$
If $\Lambda$ is finite, the open unit ball $D$ of $V$ is the Cartesian product
$\mbox{\Large$\boldsymbol\times$}_\alpha D_\alpha$ of the open unit balls $D_\alpha$ of $V_\alpha$.
In particular, the polydisc $\mathbb{D}^d$ is the open unit ball of the JB*-triple
$\mathbb{C}\oplus_\infty \cdots \oplus_\infty \mathbb{C}$ of $d$-fold $\ell_\infty$-sum  of $\mathbb{C}$.
\end{example}

Let  $V$ be a JB*-triple with open unit ball $D$.
Given $a,b\in V$, we define the {\it box operator}
$a\bo b: V \longrightarrow V$ by
$$a\bo b (x) = \{a,b,x\} \qquad (x\in V)$$
where $\|a\bo b\| \leq \|a\|\|b\|$. We note that $\|(a\bo a)(a)\| = \|a\|^3$.

  The {\it Bergman operator} $B(b,c): V\longrightarrow V$, where $b,c \in V$, plays a crucial role
in our investigation. It is a linear map defined by
\begin{equation}\label{b1}
B(b,c)(x) = x - 2(b\bo c)(x)+\{b,\{c,x,c\},b\} \qquad (x\in V).
\end{equation}
We note that
\begin{equation}\label{bbcx}
\|B(b,c)(x)\| \leq \|x\| + 2\|b\|\|c\|\|x\| + \|b\|^2\|c\|^2\|x\| = (1+\|b\|^2\|c\|^2)\|x\|.
\end{equation}

For each $a\in D$, the square root $B(a,a)^{1/2}$ exists and there is an induced
 biholomorphic map $g_a : D \longrightarrow D$,
 called a {\it transvection} (or {\it M\"obius transformation}),
 given by
\begin{equation}\label{g2}
g_a(x) = a + B(a,a)^{1/2}(I+x\bo a)^{-1}(x) \qquad (x\in D)
\end{equation}
where $I$ denotes the identity operator on $V$ and the inverse
$(I+x\bo a)^{-1}: V \longrightarrow V$ exists as $\|x\bo a\|\leq \|x\|\|a\| <1$.
In particular, $g_0$ is the identity map on $D$ and is the only transvection
having a fixed point in $D$. We note that
$g_a(0)=a$ and $g_a^{-1} = g_{-a}$ \cite[Lemma 3.2.25]{book2}.
In the Cartan decomposition associated to $D$ in (\ref{rank}), $\frak g$ is the Lie
algebra of the automorphism group of $D$ and
$\exp (\frak p)$ is the set of transvections.

On the unit disc $\mathbb{D}$, we have
\begin{equation}\label{bz}
B(a,a)z=( |1-|a|^2)^2 z \quad {\rm and} \quad B(a,a)^{1/2}z=( |1-|a|^2) z \quad (z\in \mathbb{C}).
\end{equation}

 The Kobayashi distance $\kappa$ on $D$ cam be computed using transvections:
\begin{equation}\label{alphapho}
\kappa(z,  w)= \tanh^{-1}\|g_{-z}(w)\| = \kappa(w,z) \qquad ( z, w\in D).
\end{equation} The distance $\kappa$ is contracted by holomorphic maps and
$\|w\|= \tanh \kappa(0,w) $ for $w\in D$ \cite[Theorem 3.5.9]{book2}.

The quadratic operators $Q_a : V \longrightarrow V$, where $a\in V$, also play an important role in
JB*-triples. These operators are defined by
$$Q_a(x) = \{a,x,a\} \qquad (x\in V).$$

We note that $B(a,b)$ is invertible for $\|a\|\|b\|<1$ (cf.\,\cite[(7.3) vii]{LOOS}).
 The proof of the following two identities
 can be found in \cite[Proposition 3.2.13. Lemma 3.2.17]{book}.
\begin{equation}\label{bid2}
\|B(z,z)^{-1/2}\| = \frac{1}{1-\|z\|^2} \qquad (\|z\|<1).
\end{equation}
\begin{equation}\label{bid}
1- \|g_{-y}(z)\|^2 = \frac{1}{\|B(z,z)^{-1/2}B(z, y)B(y, y)^{-1/2}\|}
\qquad (y, z\in D).                                                                                                                                                                      
\end{equation}

For a Hilbert ball $D$, this formula simplifies to
\begin{equation}\label{hh}
1-\|g_{-b}(a)\|^2 = \frac{(1-\|a\|^2)(1-\|b\|^2)}{|1-\langle a,b\rangle|^2} \qquad (a,b\in D)
\end{equation}
(cf.\,\cite[Example 3.2.29]{book2})
where $\langle\cdot,\cdot\rangle$ denotes the inner product of the underlying Hilbert space $V$.

 \begin{lemma}\label{hhh} Let $D$ be  the open unit ball of a JB*-triple $V$.
Given a sequence $(y_m)$ in $D$  norm converging to a boundary point $\xi \in \partial D$, we have
\begin{equation}\label{z}
\lim_{m\rightarrow \infty}\|g_{-y_m}(z)\| = \lim_{m\rightarrow \infty} \|g_{-z}(y_m)\|=1 \qquad (z\in D).
\end{equation}
\end{lemma}
\begin{proof}
By (\ref{bid2}) and (\ref{bid}), we have
\begin{eqnarray*}
0 &<& 1- \|g_{-y_m}(z)\|^2 = \frac{1}{\|B(z,z)^{-1/2}B(z, y_m)B(y_m, y_m)^{-1/2}\|}\\
&\leq& \frac{\|B(z, y_m)^{-1}B(z, z)^{1/2}\|}{\|B(y_m, y_m)^{-1/2}\|}
= \|B(z, y_m)^{-1}B(z, z)^{1/2}\|(1-\|y_m\|^2) \longrightarrow 0 \qquad (z\in D)
\end{eqnarray*}
as $m \rightarrow \infty$, where $\lim_m \|B(z, y_m)^{-1}B(z, z)^{1/2}\|=
\|B(z,\xi)^{-1}B(z, z)^{1/2}\|$.
\end{proof}

An element $e$ in a JB*-triple $V$ is called a {\it tripotent} if $\{e,e,e\}=e$.
Tripotents in C*-algebras are exactly the partial isometries.
Since $\|e\|= \|e\|^3$, we have $\|e\|=1$ if (and only if) $e \neq 0$.
The nonzero tripotents in $V$ are exactly the limit points $\lim_{t\rightarrow \infty} \exp_0(tv)$
$(v\in V)$
of the geodesic rays in its open unit ball $D$ emanating from $0\in D$ (cf.\,\cite[4.8]{LOOS})
and they form a real analytic manifold in the norm topology (cf.\,\cite[4.4]{sau}), denoted by $T(D)$.

A tripotent $e$ in a JB*-triple $V$ induces an eigenspace decomposition of $V$, called the
{\it Peirce decomposition}. The eigenvalues of the box operator $e\bo e: V \to V$ are in
the set $\{0, 1/2,1\}$. Let
$$V_k(e) = \left\{x\in V: (e\bo e)(x)=\frac{k}{2} x\right\} \qquad (k=0,1,2)$$
be the eigenspaces, called the {\it Peirce $k$-space} of $e$. Then we have the algebraic direct sum
$$V= V_0(e) \oplus V_1(e) \oplus V_2(e)$$
(which is usually not an $\ell_\infty$-sum). The Peirce $k$-space $V_k(e)$ is the range of
the {\it Peirce $k$-projection} $P_k(e) : V \longrightarrow V$, given by
$$P_2(e) = Q_e^2, \quad P_1(e) = 2(e\bo e -Q_e^2), \quad P_0(e) = B(e,e)$$
which are contractive projections \cite[Proposition 3.2.31]{book2}.
Each Peirce $k$-space $V_k(e)$ is a JB*-triple in the inherited norm and triple product from $V$.
We note from (\ref{jp1}) that
\begin{equation}\label{jp2}
\{e, V_j(e), e\} = \{e, P_2(e)(V_j(e)),e\} = \{0\} \qquad (j=0,1).
\end{equation}
Also, $P_2(e)$ is  a {\it structural} projection
in the sense of Loos (cf. \cite[p. 206]{em}). Hence $P_1(e)=0$ if and only if $e\bo e : V \longrightarrow V$ is
a structural projection.

Nonzero tripotents in a JB*-triple $V$ can be classified according to their Perice $k$-spaces.

\begin{xdefn}
A nonzero tripotent $e$ in a JB*-triple $V$  is called {\it maximal} if $V_0(e)=\{0\}$.
It is called {\it structural} if $V_1(e)=\{0\}$. We call $e$ a {\it minimal tripotent} if $V_2(e)= \mathbb{C}e$,
which is equivalent to $\{e, V,e\}= \mathbb{C}e$.  The maximal tripotents in $V$, if they exist,
are exactly the extreme points of the closure $\overline D$ of its open unit ball $D$ \cite[Theorem 3.2.3]{book}.
We define $$T_0(D)=\{e\in T(D): V_0(e)=\{0\}\} \quad {\rm and} \quad T_1(D)=\{e\in T(D): V_1(e)= \{0\}\}.$$
\end{xdefn}
We note that $T_0(D) \cap T_1(D) = \{e\in T(D): V_2(e)=V\}$ is the set of {\it unitary} tripotents.
The following lemma shows that $T_0(D)$, $T_1(D)$ and $T_0(D) \cup T_1(D)$  are submanifolds of $T(D)$.

\begin{lemma} The sets $T_0(D)$ and $T_1(D)$ are open in $T(D)$.
\end{lemma}
\begin{proof}
For $k=0,1$, we show that $T(D) \backslash T_k(D)$ is closed in $T(D)$.
Let $(e_n)$ be a sequence in $T(D) \backslash T_k(D)$ converging to $e\in T(D)$.
Then we have $P_k(e_n) \neq 0$ for all $n$. If $V_k(e)=\{0\}$, then continuity of the triple product implies
$$0= \|P_k(e)\| = \lim_{n\rightarrow \infty} \|P_k(e_n)\|$$
which gives the contradiction that $1=\|P_k(e_n)\| <1/2$ from some $n$ onwards.
Hence $V_k(e) \neq \{0\}$ and $e \in T(D) \backslash T_k(D)$.
\end{proof}

Given $a,b\in V$, we say that $a$ and $b$ are {\it orthogonal} to each other if  $a \bo b =0$,
in which case we have $b\bo a =0$ and
\begin{equation}\label{osum}
\|a+ b\|= \max\{\|a\|, \|b\|\}
\end{equation}
from \cite[Corollary 3.1.21]{book}. Orthogonality of two tripotents $e,c\in V$ is equivalent to
$c\in V_0(e)$, that is, $\{e,e,c\}=0$ \cite[Corollary 1.2.46]{book}.
We also  have
$$V_2(e) \bo V_0(e) := \{ a\bo b: a\in V_2(e), b\in V_0(e)\} = \{0\}$$
\cite[Theorem 1.2.44]{book}. Given $a= e_1 + e_2$ where $e_1$
and $e_2$ are orthogonal tripotents, and given $b\bo a=0$, then we also have $b\bo e_1 = b \bo e_2 =0$ since
$$\{e_1,e_1,b\} =\{e_1, a,b\}=0 =\{e_2,e_2,b\}$$
and hence $b \in V_0(e_1) \cap V_0(e_2)$ as well as $e_1 \bo b= e_2 \bo b =0$.

A minimal tripotent $c$ can never be a sum of two mutually orthogonal nonzero tripotents $u$ and $v$.
Indeed, if $c= u+v$ and $u\bo v=0$, then $\alpha c = \{c,u,c\} = \{u+v, u, u+v\} = u$ for some
$\alpha \in \mathbb{C}$ and similarly, $v= \beta c$ for some $\beta \in \mathbb{C}$, which implies
$\alpha \overline \beta c \bo c = u\bo v=0$ and $\alpha \overline \beta=0$, resulting in $u=0$ or $v=0$.

\begin{lemma}\label{24} Let $e$  be a minimal tripotent in a JB-triple $V$ such that $V_1(e)=\{0\}$.
Then for each minimal tripotent $c\in V$, we have either $c=\lambda e$ for some $|\lambda|=1$ or
$c \bo e=0$.
\end{lemma}
\begin{proof} By assumption, $V$ has a Peirce decomposition $V= \mathbb{C}e + V_0(e)$.
Hence a minimal tripotent $c$ has a decomposition $c= \lambda e + u$ for some $\lambda \in \mathbb{C}$ and
$u\in V_0(e)$.
By orthogonality, we have
$$\lambda e + u = c = \{c,c,c\} = |\lambda|^2\lambda e + \{u,u,u\}$$
which implies $u=\{u,u,u\}$ and $|\lambda|=1$ if $\lambda \neq 0$, in which case $\lambda e$ is a
nonzero tripotent.
By the preceding remark, we have either $c=\lambda e$ or $c=u \in V_0(e)$.
\end{proof}

\begin{lemma}\label{c(cee)c}In a JB*-triple $V$, let $c=  c_1+c_2$ be an orthogonal sum of tripotents.
 Let $e\in V$ with $P_2(c_1)(e)= \lambda c_1$ for some $\lambda \in \mathbb{C}$. Write
$e^1 = P_1(c_1)(e)$. Then we have
\begin{equation}\label{cee1}
P_2(c_1)\{c,e,c\} = \overline\lambda  c_1
\end{equation}
and
\begin{equation*}
P_2(c_1)( \{c, e,e\})= |\lambda|^2  c_1 +  \{c_1,\{c_1, e^1,e^1\}, c_1\}.
\end{equation*}
\end{lemma}

\begin{proof}
We note that $c\bo c_1 =  c_1\bo c_1$, $P_0(c_1)(e) \bo c_1 =0$  and
$\{e,c_1,c_1\}=\{ P_2(c_1)(e) + P_1(c_1)(e),\, c_1,c_1\} = \lambda c_1 + \frac{e^1}{2}$.
 By the triple identity, we have
\begin{eqnarray*}
\{c_1, \{c,e,c\}, c_1\}&=& 2\{\{e,c,c_1\}, c,c_1\} - \{e,c, \{c_1,c, c_1\}\}\\
&=& 2 \{\{e,c_1,c_1\}, c_1,c_1\} -  \{e,c_1,c_1\}\\
&=& 2\{ \lambda c_1 + \frac{e^1}{2},\,  c_1,c_1\} -  \lambda c_1 - \frac{ e^1}{2}\\
&=&  2 \lambda c_1 + \frac{ e^1}{2}-  \lambda c_1 - \frac{ e^1}{2} = \lambda c_1
\end{eqnarray*}
which proves $P_2(c_1)\{c,e,c\} = \overline\lambda  c_1$.
Next, the triple identity gives
\begin{eqnarray*}
\{c_1, \{c,e,e\}, c_1\}&=& 2\{\{e,c,c_1\}, e,c_1\} - \{e,c, \{c_1,e, c_1\}\}\\
&=&  2 \{\{e,c_1,c_1\},\, e, \,c_1\} -  \overline \lambda \{e, c_1, c_1\}\\
&=& 2 |\lambda|^2c_1 + \frac{\overline\lambda e^1}{2} + \{e^1,e^1,c_1\}
 -  |\lambda|^2 c_1 -\frac{ \overline\lambda e^1}{2}\\
&=& |\lambda|^2c_1 +  \{e^1,e^1,c_1\}
\end{eqnarray*}
where
$$\{e^1,e,c_1\}= \{e^1,\, \lambda c_1 + e^1,\, c_1\}= \frac{\overline\lambda e^1}{2} + \{e^1,e^1,c_1\}.$$
It follows  that
\begin{eqnarray*}
P_2(c_1)( \{c, e,e\}) &=& \{c_1, (  |\lambda|^2c_1 +  \{e^1,e^1,c_1\} ),\, c_1\}\\
&=&  |\lambda|^2  c_1 +  \{c_1,\{c_1, e^1,e^1\}, c_1\}
\end{eqnarray*}
which proves the second identity.

\end{proof}

 Let $\{e_{1},\dots,e_{n}\}$ a family  of mutually orthogonal tripotents in a JB*-triple $V$. For $i,j \in \{0,1, \ldots,n\}$,
 the {\it joint} Peirce space $V_{ij}$ is defined by
\begin{eqnarray*}\label{ij}
V_{ij}:=  V_{ij}(e_{1},\dots,e_{n})  =\{z\in V\medspace:\medspace2\{e_{k},e_{k},z\}=(\delta_{ik}+\delta_{jk})z\medspace\text{for }k=1,\dots,n\}
\end{eqnarray*}
where $\delta_{ik}, \delta_{jk}$ are the Kronecker delta and $V_{ij}=V_{ji}$. If $1\leq i < j\leq n$, then we have
\begin{equation}\label{vij}
V_{ij} = V_1(e_i) \cap V_1(e_j).
\end{equation}
We also have
\begin{equation}\label{io}
V_{0i}=V_1(e_i) \cap \bigcap_{j\neq i} V_0(e_j) \qquad (i=1, \ldots, n).
\end{equation}

The decomposition
\begin{eqnarray*}
V & = & \sum_{0\leq i\leq j\leq n}V_{ij}
\end{eqnarray*}
is called a {\it joint} Peirce decomposition, which is an algebraic direct sum.

The Peirce multiplication rules
\begin{eqnarray}\label{bo}
\{V_{ij},V_{jk},V_{km}\}  \subset  V_{im}\q {\rm and} \q V_{ij} \bo V_{pq} =\{0\} \q {\rm for} \q
\{i,j\}\cap \{p,q\}=\emptyset
\end{eqnarray}
hold. For a singleton $\{e_1\}$, we have
$$\{V_i(e_1), V_j(e_1), V_k(e_1)\} \subset V_{i-j+k}(e_1), \quad {\rm where}\quad V_p(e_1)=\{0\} \quad
{\rm for}\quad p\notin \{0,1,2\}.$$

  The contractive projection $P_{ij}(e_{1},\dots,e_{n})$ from $V$ onto $V_{ij}(e_{1},\dots,e_{n})$
is called a {\it joint Peirce projection} which satisfies
\begin{equation}
 P_{ij}(e_{1},\dots,e_{n})(e_{k})= \left\{\begin{array}{ll} 0& (i \neq j)\\
                                                   \delta_{ik} e_k & (i=j).\label{pijek}
                                         \end{array}\right.
\end{equation}
We shall simplify the notation $P_{ij}(e_1,\ldots,e_n)$ to $P_{ij}$ if the tripotents $e_1, \ldots, e_n$
are understood. Note that $P_{ij}=P_{ji}$.

Let $M=\{0,1,\dots,n\}$ and $N\subset\{1,\dots,n\}$. The Peirce $k$-spaces of the tripotent $e_{N}=\sum_{i\in N}e_{i}$
 are given by
\begin{eqnarray}
V_{2}(e_{N}) & = & \sum_{i,j\in N}V_{ij},\label{Peirce-2 of sum}\\
V_{1}(e_{N}) & = & \sum_{\substack{i\in N\\
j\in M\backslash N
}
}V_{ij},\\
V_{0}(e_{N}) & = & \sum_{i,j\in M\backslash N}V_{ij}.\label{eq:Peirce-0 of sum}
\end{eqnarray}
In particular, $P_2(e_i) =P_{ii}$. By \cite[Lemma 2.1]{horo}, we have
\begin{equation}\label{lemma 2.1}
P_{ij}(\{e_k: k\in N\}) = P_{ij}(e_1, \ldots, e_n) \quad {\rm for} \quad i,j\in N,
\end{equation}
\begin{equation}\label{221}
P_{0j}(\{e_k: k\in N\}) = \sum_{i\in M \backslash N} P_{ij}(e_1, \ldots, e_n) \quad {\rm for} \quad j\in N,
\end{equation}
\begin{equation}\label{p00}
P_{00}(\{e_k: k\in N\}) = \sum_{i\leq j\,(i,j\in M \backslash N)} P_{ij}(e_1, \ldots, e_n).
\end{equation}

The Peirce projections provide a very useful formulation of the Bergman
operators.
Let $e_{1},\dots,e_{n}$
be  mutually orthogonal tripotents in a JB*-triple $V$ and
let $x=\sum_{i=1}^{n}\lambda_{i}e_{i}$ with
$\lambda_{i}\in\mathbb{D}$.
Then the Bergman
operator $B(x,x)$ satisfies
\begin{eqnarray}
B(x,x) & = & \sum_{0\leq i\leq j\leq n}(1-|\lambda_{i}|^{2})(1-|\lambda_{j}|^{2})P_{ij}\label{eq:Loos Bergman}
\end{eqnarray}
where we set $\lambda_{0}=0$ and $P_{ij}=P_{ij}(e_1, \ldots,e_n)$ (cf.\,\cite[Corollary 3.15]{LOOS}).
This gives the following formulae
for the square roots
\begin{eqnarray}
B(x,x)^{1/2} & = & \sum_{0\leq i\leq j\leq n}(1-|\lambda_{i}|^{2})^{1/2}(1-|\lambda_{j}|^{2})^{1/2}P_{ij}\label{eq:Bergman Sq Rt}\\
B(x,x)^{-1/2} & = & \sum_{0\leq i\leq j\leq n}(1-|\lambda_{i}|^{2})^{-1/2}(1-|\lambda_{j}|^{2})^{-1/2}P_{ij}.
\label{eq:Bergman Neg Sq Rt}
\end{eqnarray}
For arbitrary $x$ in the open unit ball $D \subset V$, Loos has shown in \cite[Corollary]{loos} the  remarkable formula
$$B(x,x)^{1/2} = B(\psi(x,x),x) $$
where $\psi: \overline D \times \overline D \longrightarrow
\overline D$ is a continuous map constructed using binomial series.

\begin{lemma}\label{baab1}
Let $a$ and $b$ be two mutually orthogonal elements in a JB*-triple $V$ with open unit ball $D$. Then
$B(c,a)(b) = b$ for all $c\in V$ and $B(a,a)^{1/2}(b)=b$ if $a\in D$.
\end{lemma}

 \begin{proof}
Evidently, $a \bo b=0$ implies
$$B(c,a)(b) = b - 2 \{c,a, b\} + \{c,\{a,b,a\},c\} = b.$$
It follows from  Loos's formula above that $ B(a,a)^{1/2}(b)= B(\psi(a,a), a) (b)=b$.
\end{proof}

\begin{lemma}\label{k=0}
Let $D$ be  the open unit ball of a JB*-triple $V$. Given  $a,b, c\in D$
such that $a \bo b=0$ and $c+b\in D$, we have
$g_{a+b} = g_a \circ g_b$ and
$$g_{_{a}}(c+b) =  g_{_{ a}}(c) + b.$$
\end{lemma}
\begin{proof}
We note that $\|a+b\| = \max \{\|a\|, \|b\|\} <1$.
The first assertion has been proved in \cite{sau}.
We have
\begin{eqnarray*}
g_{_{a}}(c+b) & = & a + B( a, a)^{1/2}(I + (c+b) \bo a)^{-1}(c+b)\\
&=& a + B( a, a)^{1/2}(I+ c\bo a)^{-1}(c+b)\\
&=& a +  B( a, a)^{1/2}(I -c\bo a + (c\bo a)^2 - \cdots )(c+b)\\
&=& a +  B( a, a)^{1/2}(I-c \bo a + (c\bo a)^2 - \cdots)(c) +  B( a, a)^{1/2}(b)\\
&=& a +  B( a, a)^{1/2}(I+ c\bo a)^{-1}(c) +  b
=  g_{_{a}}(c) +  b
\end{eqnarray*}
by Lemma \ref{baab1}
\end{proof}

\begin{lemma}\label{inc}
Let $D$ be a bounded symmetric domain and let $a\in D$ be given by
$$a = \beta_1 e_1 + \cdots \beta_p e_p \qquad (\beta_1, \ldots, \beta_p \in \mathbb{D})$$
where $e_1, \ldots, e_p$ are mutually orthogonal tripotents. Then for $\alpha_1, \ldots, \alpha_p \in \mathbb{D}$,
we have
$$g_a(\alpha_1e_1 + \cdots + \alpha_p e_p) = \psi_{\beta_1}(\alpha_1)e_1 + \cdots + \psi_{\beta_p}(\alpha_p)e_p$$
where  $\psi_{\beta_j} : \mathbb{D}\longrightarrow \mathbb{D}$, for $j=1, \ldots, p$,
are transvections induced  $\beta_j \in \mathbb{D}$, with
$$\psi_{\beta_j}(z) = \frac{\beta_j +z}{1+ \overline\beta_j z} \qquad (z\in \mathbb{D}).$$
\end{lemma}
\begin{proof} First, we have
\begin{eqnarray*}
&&g_{_{\beta_1 e_1}}(\alpha_1 e_1)  =  \beta_1 e_1
+ B(\beta_1 e_1 , \beta_1 e_1 )^{1/2}(I + \alpha_1 e_1 \bo \beta_1 e_1)^{-1}(\alpha_1 e_1)\\
&=& \beta_1 e_1
+ B(\beta_1 e_1 , \beta_1 e_1 )^{1/2}(I +\alpha_1 \overline \beta_1 e_1\bo e_1)^{-1}(\alpha_1 e_1)\\
&=& \beta_1 e_1 + B(\beta_1 e_1 , \beta_1 e_1 )^{1/2}(I -  \alpha_1\overline\beta_1  e_1 \bo e_1
+ \alpha_1^2 \overline\beta_1^2 (e_1 \bo e_1)^2 - \cdots)(\alpha_1 e_1)\\
&=&\beta_1 e_1 + B(\beta_1 e_1 , \beta_1 e_1 )^{1/2}(1 -  \alpha_1\overline\beta_1
+ \alpha_1^2 \overline\beta_1^2 - \cdots)(\alpha_1 e_1)\\
&=& \beta_1 e_1 + B(\beta_1 e_1 , \beta_1 e_1 )^{1/2}(1 + \alpha_1\overline\beta_1)^{-1}(\alpha_1 e_1)\\
&=& \beta_1 e_1 + \frac{\alpha_1(1-|\beta_1|^2)}{1 + \alpha_1\overline\beta_1}e_1 \qquad ({\rm by}~ (\ref{eq:Bergman Sq Rt}))\\
&=&  \frac{\beta_1+\alpha_1}{1 + \alpha_1\overline\beta_1}e_1
=  \psi_{\beta_1}(\alpha_1) e_1.
\end{eqnarray*}
Now a repeated application of Lemma \ref{k=0} gives
\begin{eqnarray*}
&&g_a(\alpha_1e_1 + \cdots + \alpha_p e_p)
= g_{_{\beta_1 e_1}} \circ \cdots\circ g_{_{\beta_p e_p}}(\alpha_1e_1 + \cdots + \alpha_p e_p)\\
&=& g_{_{\beta_1 e_1}} \circ \cdots\circ g_{_{\beta_{p-1} e_{p-1}}}(\alpha_1e_1 + \cdots + \alpha_{p-1}e_{p-1}+
g_{_{\beta_p e_p}}(\alpha_p e_p))\\
&=&g_{_{\beta_1 e_1}} \circ \cdots\circ g_{_{\beta_{p-1} e_{p-1}}}(\alpha_1e_1 + \cdots + \alpha_{p-1}e_{p-1}+
\psi_{\beta_p }(\alpha_p) e_p)\\
&=&g_{_{\beta_1 e_1}} \circ \cdots\circ g_{_{\beta_{p-2} e_{p-2}}}(\alpha_1e_1 + \cdots + g_{_{\beta_{p-1}e_{p-1}}}(\alpha_{p-1}e_{p-1})+
\psi_{\beta_p }(\alpha_p) e_p)\\
&=& g_{_{\beta_1 e_1}} \circ \cdots \circ g_{_{\beta_{p-2} e_{p-2}}}(\alpha_1e_1 + \cdots + \psi_{\beta_{p-1}}(\alpha_{p-1})e_{p-1}+
\psi_{\beta_p }(\alpha_p) e_p)\\
&=& \psi_{\beta_1}(\alpha_1)e_1 + \cdots + \psi_{\beta_p}(\alpha_p)e_p.
\end{eqnarray*}
\end{proof}

Let $D$ be a bounded symmetric domain realised as the open unit ball of a JB*-triple $V$.
A vector subspace  $E$ of $V$ is called a {\it subtriple} if $a,b,c \in E$ implies
$\{a,b,c\} \in E$. The {\it rank} of $D$,
and of $V$, is defined to be the extended real number
$$\sup\{ \dim V(a): a\in V\} \in \mathbb{N}\cup \{\infty\}$$
where $V(a)$ denotes the smallest closed subtriple of $V$ containing $a$.
For instance, the polydisc $\mathbb{D}^d$ is of rank $d$.
In finite dimensions,
the rank is the (common) dimension of a maximal abelian subspace of $\frak p$ in the Cartan decomposition
(\ref{rank}) associated with $D$ (cf.\,\cite[p.\,185]{book2}).

We say that $V$, or $D$, is of {\it finite rank} if it is of rank $r \in \mathbb{N}$, in which case its rank is the cardinality of a maximal  family of
mutually orthogonal minimal tripotents \cite[p.\,187]{book2}. Finite-rank JB*-triples are reflexive Banach spaces
\cite[Theorem 3.3.5]{book2} and hence their open unit ball $D$ is relatively weakly compact, that is,
the norm closure $\overline D$ is compact in the weak topology $w(V,V^*)$.

The {\it rank} of a tripotent $e\in V$ is defined to be the rank
of the Peirce 2-space $V_2(e)$. If $e$ is a finite sum $ e=e_1 + \cdots +e_p$ of pairwise orthogonal minimal tripotents,
then $e_1, \ldots, e_p \in V_2(e)$ and the rank of $e$ is $p$.
If $V$ is of finite rank $r$ and $p=r$, then  $V_0(e) =\{0\}$ and
$e$ is a maximal tripotent in $V$.

\'E. Cartan's classification of finite dimensional bounded symmetric domains in \cite{cartan} can be extended to
finite-rank domains \cite[Theorem 3.3.5;(3.31)]{book2}.
 The irreducible ones are (biholomorphic to) the open unit balls of the following six types of JB*-triples,
 called {\it Cartan factors}:
\begin{equation*}
\begin{aligned}
{\rm (I)} &\q {L}(\mathbb{C}^d,K)\q (d =1,2, \ldots),~ {\rm rank}\,= d\leq \dim K,\q \\
\text{\rm (II)} &\q
\{z\in {L}(\mathbb{C}^d): z^t=-z\}\q (d = 5,6,  \ldots),~ {\rm rank}\,=\left[\frac{d}{2}\right]\\
\text{\rm (III)} &\q  \{z\in
{L}(\mathbb{C}^d): z^t=z\} \q (d = 2,3,  \ldots),~ {\rm rank}\,= d\\
\text{\rm (IV)} &\q \text{\rm spin
factor},~ {\rm rank}\,= 2\\
\text{\rm (V)} &\q M_{1,2}(\mathcal{O}),~ {\rm rank}\,= 2\\
\text{\rm (VI)} &\q H_3(\mathcal{O}),~ {\rm rank}\,= 3
\end{aligned}
\end{equation*}
where ${L}(\mathbb{C}^d,K)$ is the JB*-triple of bounded linear
operators from $\mathbb{C}^d$ to a Hilbert space $K$ and
$z^t$ denotes
the transpose of $z$ in the JB*-triple ${L}(\mathbb{C}^d) $
of $d \times d$ complex matrices. 
The JB*-triple $M_{1,2}(\mathcal{O})$ consists of $1 \times 2$ matrices over the Cayley algebra $\mathcal{O}$,
and $H_3(\mathcal{O})$ consists of $3 \times 3$ hermitian matrices over $\mathcal{O}$.

A {\it spin factor} is a JB*-triple $(V, \{\cdot,\cdot,\cdot\})$, with $\dim V>2$,  equipped with a complete inner product $\langle \cdot,\cdot\rangle$
and a conjugate linear isometry $^*: V \longrightarrow V$ satisfying
\begin{enumerate}
\item[\rm (i)] $x^{**} = x$,~ $\langle x^*, y^*\rangle = \langle y, x\rangle$,
\item[\rm (ii)] $\{x,y,z\} = \frac{1}{2}(\langle x,y\rangle z+ \langle z,y\rangle x - \langle x, z^*\rangle y^*) \quad (x,y,z \in V)$.
\end{enumerate}
The open unit ball of a spin factor $V$ is known as a {\it Lie ball}.
The nonzero tripotents in  $V$ are either minimal or maximal, and the minimal tripotents are exactly
the extreme points $e$ of the closed unit ball of the Hilbert space $(V, \langle \cdot,\cdot\rangle)$ satisfying
$\langle e,e^*\rangle =0$, in which case $e^*$ is also a minimal tripotent and we have
\begin{equation}\label{spin}
P_2(e)= \langle \cdot, e\rangle e, \quad P_0(e)= \langle \cdot, e^*\rangle e^*. 
\end{equation}
We refer to \cite[3.3]{book2} for more details.

Evidently, the rank of a finite dimensional domain
is finite although an infinite dimensional bounded symmetric domain can be of finite rank,
as shown in the examples of (I) and (IV) above.

For a finite-rank bounded symmetric domain $D$ realised as the open unit ball of a JB*-triple, it has been shown
in \cite[Theorem 3.4.6]{book2} that its Shilov boundary $\Sigma(D)$ is exactly the manifold $T_0(D)$ of maximal tripotents
in $\partial D$, which is nonempty.

\begin{example}\label{sigma*}
 Given a Lie ball $D$ in a spin factor $V$, we have $\Sigma(D)=T_1(D) \neq T(D)$.
Indeed, $V$ always contains minimal tripotents and they are not maximal tripotents.
Given a maximal tripotent $e\in V$, the box operator $e \bo e$ is the identity operator on $V$
\cite[Lemma 2.5.4]{book2}.
Hence $x\in V_1(e)$ implies $\frac{x}{2} = e \bo e(x) =x$ and $x=0$, proving that $e$ is a structural tripotent.
Conversely, if $e$ is a structural tripotent, then $e$ cannot be a minimal tripotent for otherwise we would have
$V = V_2(e) \oplus V_0(e) = \mathbb{C}e + \mathbb{C}e^*$ and $\dim V =2$. Hence $e$ must be a maximal tripotent.
This proves $\Sigma(D)=T_1(D)$.
\end{example}

An important property of a finite-rank JB*-triple $V$ is the existence of a spectral decomposition for each
element of $V$. If $V$ is of rank $r<\infty$, then each $a\in V$ admits a spectral decomposition of the form
$$a = \alpha_1 e_1 + \alpha_2 e_2 + \cdots + \alpha_r e_r$$
where $\|a\|=\alpha_1 \geq \alpha_2 \geq \cdots \geq \alpha_r \geq 0$ and $e_1, e_2, \ldots, e_r$ are mutually
orthogonal minimal tripotents (cf.\,\cite[(3.32)]{book2}).

\begin{remark}\label{uni2}
The spectral values $\alpha_1, \ldots,  \alpha_r$ in the above spectral decomposition are unique, but the minimal tripotents $e_i$ need not be.
We can however collect terms with equal positive spectral values in the sum and write $a = \sum_{i=1}^{r'}\alpha_ic_i$, where $r'\leq r$ and  $c_i$'s are (not necessarily minimal) pairwise orthogonal tripotents. In this case,  $c_i$'s are unique  \cite[Theorem 1.2.34]{book}.
\end{remark}

\begin{remark}\label{sub}
Given a sequence $(z_m)$ in a  bounded symmetric domain $D$ of rank $r<\infty$, with spectral decomposition
$$z_m =  \alpha_{m1}e_{m1} + \cdots + \alpha_{mr}e_{mr} \quad (1>\|z_m\| = \alpha_{m1} \geq \cdots \geq \alpha_{mr}\geq 0),$$
we can, by weak compactness of $\overline D$, choose weakly convergent subsequences of the sequences
$( \alpha_{mj}e_{mj})$ successively for $j=1, \ldots, r$,  and obtain
a subsequence $(z_{m_k})$ of $(z_m)$, with spectral decomposition
$$z_{m_k} =  \alpha_{m_k,1}e_{m_k, 1} + \cdots + \alpha_{m_k, r}e_{m_k, r}
\quad (\alpha_{m_k,1} \geq \cdots \geq \alpha_{m_k,r}\geq 0)$$
such that each sequence $( \alpha_{m_k,j})$ converges to $\alpha_j \in [0,1]$ and $(e_{m_k, j})$
weakly converges to $e_j \in \overline D$.
\end{remark}

In the sequel, we will frequently make use of the following lemma which has been proved in
\cite[Lemma 5.8]{horo}.
\begin{lemma}\label{zmm} Let $D$ be the open unit ball of a JB*-triple of finite rank and let
$(z_m)$ be a sequence in $D$ norm converging to some $\xi \in \overline D$, with spectral
decomposition
$$z_m =  \alpha_{m1}e_{m1} + \cdots + \alpha_{mr}e_{mr} \quad (\|z_m\| = \alpha_{m1} \geq \cdots \geq \alpha_{mr}\geq 0)$$
such that each sequence $( \alpha_{mi})$ converges to $\alpha_i\in [0,1] $
and  $(e_{mi})$
 weakly converges to $e_i\in \overline D$, for $i=1, \ldots, r$.
Then there is some $p\in \{1, \ldots, r\}$ such that
\begin{enumerate}
\item[(i)] $\alpha_i >0$ and $(e_{mi})_m$ norm converges to $e_i\,$, for $i = 1, \ldots, p$,
\item[(ii)] $\alpha_i=0$ for $i \geq p+1$,
\item[(iii)]$e_1, \ldots, e_p$ are mutually orthogonal minimal tripotents
\end{enumerate}
and $\xi = \alpha_1 e_1 + \cdots + \alpha_p e_p$.
\end{lemma}

\begin{definition}\label{abelian}
A JB*-triple $V$ is called {\it abelian} if
$$\{a,b,\,\{x,y,z\}\} =\{a,\,\{b,x,y\},\, z\}\} =\{\{a,b,x\},\, y,z\}\}  \qquad (a,b,x,y,z \in V).$$
We call a bounded symmetric domain {\it abelian} if it is biholomorphic to the open unit ball of
an abelian JB*-triple.
\end{definition}

Abelian C*-algebras are abelian JB*-triples. In particular, a polydisc
$\mathbb{D}^d$ is abelian as it is the open unit ball of the $d$-fold $\ell_\infty$-sum of $\mathbb{C}$.
Given a tripotent $e$ in an abelian JB*-triple $V$, we have $V_1(e)=\{0\}$ \cite[Lemma 1.2.38]{book}.
Hence we have $T(D)=T_1(D)$ for an abelian bounded symmetric domain $D$.

\begin{lemma}\label{sum} Let $V$ be an abelian JB*-triple. Given two minimal tripotents $c$ and $e$ in $V$,
we have either $c= \lambda e$ for some $|\lambda|=1$ or $c\in V_0(e)$.
Let $\{e_1, \ldots, e_n\}$ and $\{c_1, \ldots, c_m\}$ be two sets of mutually orthogonal
minimal tripotents such that
$$e_1 + \cdots + e_n = \gamma_1c_1 + \cdots +\gamma_m c_m \qquad (\gamma_1, \dots,\gamma_m \in \mathbb{C}\backslash\{0\}).$$
Then $n=m$ and there is a permutation $\tau$ of $\{1, \ldots, n\}$ such that $e_i = \gamma_{\tau(i)} c_{\tau(i)}$
 and  $|\gamma_i|=1$ for all $i$.
 \end{lemma}
\begin{proof}
Since $V$ is abelian, we have $V_1(e)=\{0\}$ and
the first assertion follows from Lemma \ref{24}.

Given the two orthogonal sums  above, we have
\begin{eqnarray*}
&&e_1 + \cdots + e_n= \{e_1 + \cdots + e_n,\, e_1 + \cdots + e_n,\, e_1 + \cdots + e_n\}\\
&=& \{ \gamma_1c_1 + \cdots +\gamma_m c_m, \, \gamma_1c_1 + \cdots +\gamma_m c_m,
\, \gamma_1c_1 + \cdots +\gamma_m c_m \}\\
&=&  |\gamma_1|^2\gamma_1c_1 + \cdots +|\gamma_m|^2\gamma_m c_m = \gamma_1c_1 + \cdots +\gamma_m c_m
\end{eqnarray*}
which implies  $|\gamma_i|=1$ for $i=1, \ldots, m$.
Therefore both sides of the above  represent the same tripotent of rank $n=m$.

On the other hand,
$$e_1 = e_1 \bo e_1(e_1+ \cdots+e_n) = \gamma_1 e_1 \bo e_1(c_1) + \cdots + \gamma_n e_1 \bo e_1(c_n)$$
implies $\gamma_j e_1 \bo e_1(c_j) \neq 0$ for some $j\in \{1, \ldots, n\}$.
In this case, the first assertion entails $e_1 = \lambda_jc_j$ for some $|\lambda_j|=1$
and hence $c_j \bo (e_2 + \cdots + e_n) =0$.
Applying $P_2(c_j)$ to both sides of the given identity yields
$\lambda_j c_j =\gamma_j c_j$ and so $\lambda_j = \gamma_j$.

 Likewise
$e_2 = \gamma_{j'} c_{j'}$ form some $j'$, and $e_1 \bo e_2 =0$ implies $j'\neq j$.
A similar argument applies to $e_3, \ldots, e_n$.
 Hence one can define
a permutation $\tau$ on $\{1, \ldots ,n\}$ such that $e_i = \gamma_{\tau (i)}c_{\tau(i)}$.
\end{proof}

\begin{xrem} Without the abelian assumption, the preceding lemma is false. Let $V= L(\mathbb{C}^2)$
be the JB*-triple of $2\times 2$ complex matrices. Then we have the following two orthogonal sums of  minimal
tripotents:
\[ \left(\begin{matrix} 1& 0\vspace{.1in}\\ 0&0 \end{matrix}\right) +
 \left(\begin{matrix} 0& 0\vspace{.1in}\\ 0&1 \end{matrix}\right) =
\left(\begin{matrix} \frac{1}{2}& \frac{1}{2}\vspace{.1in}\\ \frac{1}{2}&\frac{1}{2} \end{matrix}\right) +
\left(\begin{matrix}\frac{1}{2}& -\frac{1}{2}\vspace{.1in}\\ -\frac{1}{2}&\frac{1}{2} \end{matrix}\right). \]
\end{xrem}

To conclude this section, we note that
polydiscs are the {\it only} finite-rank abelian  bounded symmetric domains.
Indeed, let $D$ be a finite-rank bounded symmetric domain realised as the open unit ball of a finite-rank
abelian JB*-triple $V$. By the classification theorem in \cite[Theorem 3.3.5]{book2}, $V$ decomposes into
a finite $\ell_\infty$-sum $V=V_1 \oplus_\infty \cdots \oplus_\infty V_d$ of finite-rank Cartan factors, where
each subtriple $V_j$ of $V$ is abelian.  Pick a minimal
tripotent $e$ in any Cartan factor $V_j$. Then the Peirce $1$-space $(V_j)_1(e)$ in $V_j$ vanishes and it follows
from \cite[p.\,311(Case $0$)]{book2} that $V_j=\mathbb{C}e$.
Therefore $V \approx \mathbb{C}\oplus_\infty \cdots \oplus_\infty \mathbb{C}$.

\section{Boundary structures}

To generalise the Denjoy-Wolff theorem, we will show that all limit points of every
orbit $\mathcal{O}(a)$ for a  fixed-point free compact holomorphic self-map $f$ on a
finite-rank bounded symmetric domain $D$ are
contained in one single {\it holomorphic boundary component} in $\partial D$,
provided that the extended Shilov boundary of $D$ contains sufficient orbit limit points.
We first describe these
components and their properties in this section.

\begin{definition}\label{holbdc} Let $U$ be a convex domain, with closure $\overline U$,
 in a complex Banach space $V$.
A subset $\Gamma\subset \overline U$ is called a
 {\it holomorphic
component}
if the following conditions
are satisfied:
\begin{enumerate}
\item[(i)] $\Gamma \neq \emptyset$; \item[(ii)] for each holomorphic map $ \gamma: \mathbb{D}\longrightarrow \overline U$,
either $\gamma(\mathbb{D})\subset \Gamma$ or $\gamma(\mathbb{D})
\subset \overline U \backslash \Gamma$; \item[(iii)] $\Gamma$ is minimal
with respect to (i) and (ii).
\end{enumerate}
\end{definition}
Two holomorphic components  are either equal or
disjoint. The domain $U$ is the unique open holomorphic component of $\overline U$, all others
are contained in the boundary $\partial U$ \cite{ks}. The components in $\partial U$ form a
partition of $\partial U$ and are called the {\it holomorphic boundary components} of $U$.

For each $a \in \overline U$, we denote by $\Gamma_a$ the holomorphic
component containing $a$.
Let $D$ be a bounded symmetric domain realised as the open unit ball
of a JB*-triple $V$.
Given a tripotent $e\in \partial D$, it has been shown in \cite{ks} that the holomorphic {\it boundary} component
$\Gamma_e$ containing $e$
is the convex set
\begin{equation}\label{e0} \Gamma_e= e+ (V_0(e) \cap D)\subset \partial D
\end{equation}
 where $V_0(e)$ is the Peirce $0$-space
of $e$. In particular, $\Gamma_e = \{e\}$ if $e$ is a maximal tripotent.
Further, if $D$ is of finite rank, then all holomorphic boundary components
in $\partial D$ are of this form.

It has been shown in \cite[Remark 3.4.14]{book2}
 that the norm closure  $\overline\Gamma_e$
  is a face of $\overline D$, that is, $ x, y
\in \overline \Gamma_e$ whenever $\lambda  x+ (1-\lambda)  y \in \overline\Gamma_e$ with
$0<\lambda<1$ and $x,  y\in \overline D$.

\begin{example}\label{bergph1} Let $D$ be the open unit ball of a Hilbert space $V$.
The tripotents of  $V$ are exactly the extreme points of
$\overline D$. The Peirce $0$-space $V_0(\xi)$ vanishes for each tripotent $\xi\in V$ and hence
the boundary components in $\partial D$ are the singletons $\{\xi\}$ with $\xi\in \partial D$.

Let $D=D_1 \times \cdots\times D_p$ be a Cartesian product of the open unit balls $D_j$ of
Hilbert spaces $V_j$
for $j=1, \ldots, p$. Then $D$ is the open unit ball of the JB*-triple
$$V= V_1 \oplus_{\infty} \cdots \oplus_{\infty} V_p.$$
A tripotent $\boldsymbol e \in V$ is of the form $\boldsymbol e=(\xi_1, \ldots, \xi_p)$,
 where
$$\|\xi_j\| = \left\{\begin{matrix} 1 & (j\in J)\\
    0 & (j\notin J) \end{matrix}\right. $$
for some non-empty set $J \subset \{1, \ldots, p\}$.
We have
$$V_0(\boldsymbol e) =\{(v_1, \ldots, v_p): (\xi_j \bo \xi_j)(v_j)=0\,  \forall j\}=\{(v_1, \ldots, v_p): v_j=0 ~{\rm for}~
j\in J\}.$$
Hence the boundary component
 $\Gamma_{\boldsymbol e} = \boldsymbol e+ (V_0(\boldsymbol e) \cap D)$ is of the form
 $$\Gamma_{\boldsymbol e}  =  \Gamma_1 \times \cdots \times  \Gamma_p
 \quad {\rm and} \quad \Gamma_j = \left\{\begin{matrix} \{\xi_j\} & (j\in J)\\
 D_j & (j \notin J) \end{matrix} \right. $$
and $\overline{\Gamma}_{\boldsymbol e}$ is a closed face of $\overline D$.
\end{example}

\begin{xexam} Let $D$ be a Lie ball, which is the open unit ball of a spin factor $V$. As noted before,
the nonzero tripotents in $V$ are either minimal or maximal. Hence a boundary component
$\Gamma_c$ in $\partial D$ is of the form $c + \mathbb{D}c^*$ if $c$ is a minimal tripotent
since $V_0(c)=\mathbb{C}c^*$, or of the form $\{c\}$ if $c$ is a maximal tripotent.
\end{xexam}

The following useful result has been shown in \cite[Lemma 3.4.11]{book2}.

\begin{lemma}\label{3.2} Let $D$ be a domain in a Banach space and $h: D \longrightarrow \overline U$  a
holomorphic map, where $U$ is a convex domain in some Banach space. Then the image $h(D)$ is entirely contained  in
one holomorphic component $\Gamma_{h(a)}\subset \overline U$ with $a\in D$.
\end{lemma}

\section{Horoballs}\label{Horo}

Horoballs, introduced in \cite{horo}, are a crucial tool in our approach to holomorphic dynamics on bounded symmetric domains.
They are convex domains. Given a sequence $(z_k)$ in
a bounded symmetric domain $D$  norm converging to a boundary point $\xi \in \partial D$, the {\it horoball} $H(\xi, s)$ for
 $s>0$ is defined in \cite[Theorem 4.4]{horo} as
\begin{equation}\label{ball0}
H(\xi,s) = \{ x\in D: F (x) <s\}
\end{equation}
where $F: D \longrightarrow [0,\infty)$ is a continuous function given by
$$F (x) = \limsup_{k\rightarrow \infty} \frac{1-\|z_k\|^2}{1-\|g_{-x}(z_k)\|^2} =\limsup_{k\rightarrow \infty} \frac{1-\|z_k\|^2}{1-\|g_{-z_k}(x)\|^2} \qquad (x\in D)$$
in which $g_{-x}: D \longrightarrow D$ is the transvection induced by $x$.
Convexity of $H(\xi,s)$ has been shown in \cite[Theorem 4.4]{horo}.

 \begin{xrem} Actually, the horoball in (\ref{ball0}) is denoted by $H(\xi, \frac{1}{s})$ in \cite{horo}.
 We will make suitable adjustments when making use of the relevant results in \cite{horo}.
\end{xrem}

Evidently, the definition of $H(\xi,s)$ depends on the sequence $(z_k)$. Nevertheless, by choosing a subsequence, we may
assume  in later applications that $(z_k)$ satisfies an additional condition (\ref{sigi}) below. Let
$$z_k = \alpha_{k 1}e_{k1} +\cdots + \alpha_{kr}e_{kr}\qquad
(\|z_k\|=\alpha_{k1} \geq \alpha_{k2} \geq \cdots \alpha_{kr}\geq 0)$$
be a spectral decomposition.  Since $\overline D$ is weakly compact and
$$0<1-\alpha_{k1}^2 \leq 1-\alpha_{ki}^2 \qquad  (i = 1, \ldots, r),$$
we may assume, by choosing a subsequence of $(z_k)$, that the limits
\begin{equation}\label{sigi}
\alpha_i = \lim_k \alpha_{ki}, ~ \sigma_i =\lim_{k\rightarrow \infty} \frac{1-\alpha_{k1}^2}{ 1-\alpha_{ki}^2}
~~{\rm exist, ~ and~}(e_{ki}) ~{\rm weakly~ converges~to~} e_i
\end{equation}
for $i = 1, \ldots, r$. We note that $\alpha_1 =\lim_k \alpha_{k1} =\lim_k \|z_k\|=1= \sigma_1$.

Given a sequence $(z_k)$ converging to $\xi \in \partial D$ and satisfying condition (\ref{sigi}),
 we will see in Theorem \ref{horo} that  the limit
\begin{equation}\label{xi1}
 F_\xi (x) =\lim_{k\rightarrow \infty} \frac{1-\|z_k\|^2}{1-\|g_{-x}(z_k)\|^2} \qquad (x\in D)
 \end{equation}
 exists. In view of this,
 we will henceforth define, in place of (\ref{ball0}),
\begin{equation}\label{ball}
H(\xi,s) = \{ x\in D: F_\xi (x) <s\}
\end{equation}
and call it a {\it horoball defined by the sequence $(z_k)$},
with  defining function $F_\xi$ given in (\ref{xi1}). In this definition, it is understood
that $(z_k)$ satisfies (\ref{sigi}).
We note that
$$H(\xi, s) \subset H(\xi,r) \quad {\rm for}\quad 0<s<r.$$

We will show in Lemma \ref{b5} that $F_\xi(x) >0$ for all $x\in D$.
In fact,
$F_\xi$ is related to the notion of a {\it horofunction} introduced by Gromov \cite{ball,grom} for metric spaces
which are length spaces.
More precisely, $\frac{1}{2} \log F_\xi$ is a horofunction on $D$.
We explain  this connection below.

Let $(M,d)$ be a metric space and fix a base point $b\in M$. For each $z\in M$,
 we define a function $h_z: M \longrightarrow \mathbb{R}$ by
$$h_z(x) = d(x,z) - d(b,z) \qquad (x\in M).$$
Then $h_z$ is a Lipschitz $1$-function, that is,
$$|h_z(x) - h_z(y)| \leq d(x,y) \qquad (x,y\in M).$$

A {\it horofunction} on $M$ is a function $h: M \longrightarrow \mathbb{R}$ of the form
\begin{equation}\label{gro}
h(x) = \lim_{k\rightarrow \infty} h_{z_k}(x) \qquad (x\in M)
\end{equation}
where $(z_k)$ is a sequence in $M$ satisfying $\lim_k d(b, z_k) =\infty$.
We note that $h$ is also a Lipschitz $1$-function.

It may be useful to point out  the fact (although not needed later)
that if $(M,d)$ is locally compact, then one can always extract a pointwise convergent subsequence
$(h_{z_m})$ from a sequence $(h_{z_k})$ with $\lim_k d(b, z_k) =\infty$.
 This can be seen as follows.

Let $Lip_b(M)$ be the real Banach space of real-valued Lipschitz functions $f$ on $M$ satisfying $f(b)=0$, equipped with the norm
 $$\|f\|_{lip} = \sup \left\{ \frac{|f(x)-f(y)|}{d(x,y)} : x\neq y\right\}.$$
 It is well-known that $Lip_b(M)$ has a predual $\mathcal{F}(M)$. 
On the closed unit ball of $Lip_b(M)$, denoted by $Lip_b^1(M)$, the weak* topology,
 the topology $\tau_p$ of pointwise convergence as well as the topology $\tau_0$ of uniform convergence on compact sets
 in $M$ all coincide,
 and $Lip_b^1(M)$ is compact in these topologies (cf. \cite[p. 642]{jv}), which contains the horofunctions on $M$.

 Let $(M,d)$ be locally compact and equip the vector space $C(M)$ of real continuous functions on $M$
 with the topology $\tau_0$ of uniform convergence on compact sets. Then  $Lip_b^1(M)$
is a compact subset of $C(M)$.  The dual $C(M)^*$ of $C(M)$ identifies with
the space of regular Borel measures $\mu$ on $M$ with compact support, in the duality
$$ (f,\mu)\in C(M) \times C(M)^* \mapsto \int_M fd\mu.$$
It follows that  $Lip_b^1(M)$ is weakly compact in $C(M)$
since a net $(f_\alpha)$ $\tau_0$-converges to $f$ in $Lip_b^1(M)$
also converges in the weak topology in the above duality. Hence each sequence
$(f_k)$ in $Lip_b^1(M)$ admits a weakly convergent subsequence $(f_{m})$, which
is also pointwise convergent. In particular, the sequence $(h_{z_k})$ above admits a
pointwise convergent subsequence $(h_{z_m})$.

Let $D$ be a bounded symmetric domain realised as the open unit ball of a JB*-triple $V$. Then $D$ is a metric space
with the Kobayashi distance $\kappa$ as the metric.
We can define horofunctions on $(D, \kappa)$ with the origin $0\in D$
as a base point. In this setting, we show that 
 $\frac{1}{2}\log F_\xi$ is a horofunction on $(D, \kappa)$.

Let $(z_k)$ be a sequence  in $D$,  norm converging to a boundary point $\xi\in \partial D$
and satisfying (\ref{sigi}).  Then
$\lim_k \|z_k\|=1$ and hence
$$\lim_{k\rightarrow \infty} \kappa (0, z_k) =\lim_{k\rightarrow \infty} \tanh^{-1}\|z_k\| = \infty.$$

By definition, we have
 \begin{eqnarray}\label{gh}
&&h_{z_k}(x) = \kappa (x,{z_k})-\kappa (0,{z_k}) = \frac{1}{2}\log \frac{1+ \|g_{-x}({z_k})\|}{1-\|g_{-x}({z_k})\|} -
 \frac{1}{2}\log \frac{1+ \|{z_k}\|}{1-\|{z_k}\|} \nonumber\\
&=&  \frac{1}{2}\log\left[ \left(\frac{1+ \|g_{-x}({z_k})\|}{1+\|{z_k}\|}\right)\left(\frac{1- \|{z_k}\|}{1-\|g_{-x}({z_k})\|}\right)\right]\nonumber\\
&=& \frac{1}{2}\log \left[\left(\frac{1- \|{z_k}\|^2}{1-\|g_{-x}({z_k})\|^2}\right)
\left(\frac{1+\|g_{-x}({z_k})\|}{1+ \|{z_k}\|}\right)^2 \right]\qquad (x\in D)
\end{eqnarray}
where $\lim_k \left(\frac{1+\|g_{-x}({z_k})\|}{1+ \|{z_k}\|}\right)^2  =1$ by
 Lemma \ref{hhh}. Hence
\begin{equation}\label{zm}
 \lim_{m\rightarrow \infty} h_{z_k}(x)=\frac{1}{2} \log \lim_{k\rightarrow \infty} \frac{1- \|{z_k}\|^2}{1-\|g_{-x}({z_k})\|^2}
=\frac{1}{2} \log F_\xi(x) \qquad (x\in D)
\end{equation} by (\ref{xi1}) and $ \frac{1}{2} \log F_\xi$
is a horofunction on $D$. In particular, we have
\begin{equation}\label{lips}
|\log F_\xi (x) - \log F_\xi (y)| \leq 2 \kappa (x,y) \qquad (x,y \in D)
\end{equation}
as horofunctions are $1$-Lipschitz.

From now on and for later reference,
we will rename the horofunction $h$ in (\ref{gro}) a {\it Gromov horofunction}, but instead
and more appropriately,
call $F_\xi$ in (\ref{xi1})  a {\it horofunction}, which defines the horoballs $H(\xi, s)$ for $s>0$.

\begin{example}\label{h}
We note that not all Gromov horofunctions on $(D,\kappa)$ are of the form $\frac{1}{2}\log F_\xi$.
Let $D$ be the open unit ball of a Hilbert space $V$ with inner product $\langle\cdot,\cdot\rangle$,
 and let $(z_k)$  be a sequence in $D$ such that $\kappa(z_k,0) \rightarrow \infty$ as $k \rightarrow \infty$.
 Then we have $\lim_k \|z_k\| =1$, but $(z_k)$ need not be norm convergent to a boundary point in $\partial D$
 if $\dim V = \infty$.

For instance, let $(e_k)$ be an orthonormal sequence in the Hilbert space $\ell_2$ of square summable complex sequences and
  $(\alpha_k)$ a sequence in $(0,1)$ converging to $1$, then the sequence $(\alpha_k e_k)$ in $D$ converges weakly to $0$,
  by Bessel's inequality, whereas $\lim_k \|\alpha_k e_k\|=1$.

  Nevertheless,  weak compactness
  of $\overline D$ implies that $(z_k)$ admits a subsequence $(z_j)$ weakly convergent to $a\in \overline D$, but $\|a\| <1$ can occur
as shown above.

We have, by (\ref{hh}),
\begin{eqnarray*}
& &\lim_{j\rightarrow \infty} h_{z_j}(x) =\lim_{j\rightarrow \infty}
 \frac{1}{2}\log \left[ \frac{|1-\langle x, z_j\rangle|^2}{1-\|x\|^2}
 \left(\frac{1+\|g_{-x}(z_j)\|}{1+ \|z_j\|}\right)^2\right]\\
&= & \frac{1}{2}\log  \frac{|1-\langle x, a\rangle|^2}{1-\|x\|^2} \qquad (x\in D).
\end{eqnarray*}
If, for instance, $(z_k)$ is the sequence $(\alpha_ke_k) \in \ell_2$  given above, then $a=0$ and hence
$$ \lim_{j\rightarrow \infty} h_{z_j}(x) = \frac{1}{2}\log  \left(\frac{1}{1-\|x\|^2}\right)$$
which is a Gromov horofunction.

If, however, $(z_k)$ is norm convergent,  then $\|a\|=1$ and the horofunction $F_a$ is given by
$$ F_a(x)= \lim_{k\rightarrow \infty} \frac{1-\|z_k\|^2}{1-\|g_{-x}(z_k)\|^2}=   \frac{|1-\langle x, a\rangle|^2}{1-\|x\|^2}\qquad (z\in D) $$
from (\ref{hh}).
\end{example}

\begin{lemma}\label{b5}
Let $D$ be a bounded symmetric domain realised as the open unit ball of a JB*-triple $V$.
Then the horofunction $F_\xi(x)$ in (\ref{xi1}) satisfies
\begin{equation}\label{bbdd}
F_\xi(0) =1 \quad {\rm and}\quad \frac{1-\|x\|}{1+\|x\|} \leq F_\xi(x) \leq  \frac{1+\|x\|}{1-\|x\|} \qquad (x\in D).
\end{equation}
\end{lemma}

\begin{proof}
We have
$$ F_\xi (0) =\lim_{k\rightarrow \infty} \frac{1-\|z_k\|^2}{1-\|g_{0}(z_k)\|^2}
 =\lim_{m\rightarrow \infty} \frac{1-\|z_k\|^2}{1-\|z_k\|^2}=1.$$
For the inequalities,
it suffices to prove, for each  $z\in D$,  we have
$$\frac{1-\|x\|}{1+\|x\|} \leq \frac{1-\|z\|^2}{1-\|g_{-x}(z)\|^2} \leq  \frac{1+\|x\|}{1-\|x\|} \qquad (x\in D).$$
The last inequality has already been shown in \cite[Lemma 4.1]{horo}.
We deduce the first inequality using (\ref{bbcx}) and (\ref{bid}) as follows.
\begin{eqnarray*}
&& \frac{1}{1-\|z\|^2} = \frac{1}{1-\|g_x(g_{-x}(z))\|^2 }\\
&=&\|B(g_{-x}(z),g_{-x}(z))^{-1/2}B(g_{-x}(z),-x)B(-x,-x)^{-1/2}\|\\
&\leq & \frac{1}{1-\|g_{-x}(z)\|^2}(1+\|g_{-x}(z)\|\|x\|)^2 (1-\|x\|^2)^{-1}\\
&\leq & \frac{1}{1-\|g_{-x}(z)\|^2}(1+\|x\|)^2 (1-\|x\|^2)^{-1}
\leq \frac{1+\|x\|}{(1-\|g_{-x}(z)\|^2) (1-\|x\|)}.
\end{eqnarray*}
\end{proof}
Evidently, (\ref{bbdd}) implies $F_\xi(x)>0$ for all $x\in D$. It also implies
$\bigcap_{s>0} H(\xi,s) = \emptyset$ for the horoballs $H(\xi,s)$ defined by $F_\xi$  in (\ref{ball}),
since any $x\in \bigcap_{s>0} H(\xi,s)$ would lead to the contradiction that $\|x\|<1$ as well as
$(1-\|x\|)(1+\|x\|)^{-1} \leq F_\xi(x) <s$ for all $s>0$. On the other hand, we have $D = \bigcup_{s>0} H(\xi,s)$
by (\ref{xi1}).

We now establish an explicit formula for $F_\xi$ in (\ref{xi1}) for finite-rank domains $D$.

\begin{theorem}\label{horo} Let $D$ be a bounded symmetric domain of finite rank $r$, realised as the open unit
ball of a JB*-triple $V$. Then the horofunction $F_\xi$  in (\ref{xi1}), defined by   a sequence $(z_k)$ in $D$ converging to a boundary point $\xi \in \partial D$ and satisfying (\ref{sigi}), can be expressed as
\begin{equation}\label{fxi}
F_\xi(x) =\left\|\sum_{1\leq i\leq j\leq p}\rho_i\rho_jB(x,x)^{-1/2}B(x,\xi)P_{ij}\right\|
\qquad (x\in D)
\end{equation}
for some $p\in \{1, \ldots, r\}$ and $\rho_1, \ldots, \rho_p \in [0,1]$,
where  $\rho_1=1$
and  $\xi$ has a spectral decomposition
$$\xi = \alpha_1 e_1 + \alpha_2 e_2 + \cdots \alpha_p e_p \qquad (1=\alpha_1 \geq \alpha_2 \geq \cdots \geq \alpha_p>0)
$$
with joint Peirce projections $P_{ij}: V \longrightarrow V$  induced
by the minimal tripotents $e_1, \ldots, e_{p}$.
\end{theorem}

\begin{proof} Let $F_\xi$ be defined by  $(z_k)$ satisfying (\ref{sigi}), with
a spectral decomposition
$$z_k = \alpha_{k1}e_{k1} + \cdots + \alpha_{kr}e_{kr}
 \qquad (\|z_k\|=\alpha_{k1} \geq \cdots \geq \alpha_{kr} \geq 0).$$
By Lemma \ref{zmm},
 there exists $p\in \{1, \ldots, r\}$ such that
\begin{enumerate}
\item[(i)] $\alpha_i >0$ and $(e_{ki})_k$ norm converges to $e_i$ for each $i = 1, \ldots, p$,
\item[(ii)] $\alpha_i=0$ for $i >p$,
\item[(iii)]$e_1, \ldots, e_p$ are mutually orthogonal minimal tripotents
\end{enumerate}
and
$$\xi = \lim_k z_k = \alpha_1e_1 +\cdots +\alpha_p e_p \qquad (\alpha_1=1).$$

By (\ref{bid}), we have
$$ \frac{1-\|z_k\|^2}{1-\|g_{-x}(z_k)\|^2}= (1-\|z_k\|^2)\|B(x,x)^{-1/2}B(x,z_k)B(z_k,z_k)^{-1/2}\|$$
where
$$B(z_k,z_k)^{-1/2} = \sum_{0\leq i\leq j\leq r} (1-\alpha_{ki}^2)^{-1/2}(1-\alpha_{kj}^2)^{-1/2}P_{ij}^k \qquad (\alpha_{_{k0}}=0)$$
and $P_{ij}^k$ is the joint Peirce projection $P_{ij}(e_{k1}, \ldots, e_{kr})$.

For $1\leq i\leq j\leq p$, we have $P_{ij}(e_{k1}, \ldots, e_{kp}) = P_{ij}^k$ by (\ref{lemma 2.1}).
By \cite[Remark 5.9]{horo}, we have the following norm
convergence
$$\lim_{k\rightarrow \infty} P^k_{ij} = P_{ij}(e_1, \ldots, e_p).$$
Since $(z_k)$ satisfies (\ref{sigi}), the following limit
\begin{equation}\label{lambda}
\rho_i = \lim_k \sqrt{\frac{1-\alpha_{k1}^2}{1-\alpha_{ki}^2}} = \sqrt \sigma_i \in [0,1]
\end{equation}
exists for $i=0, 1, \ldots, r$, where $\rho_1 =1$ and $\rho_0=\rho_i =0$ for  $\alpha_i <1$.
In particular, $\|P_{ij}^k\|\leq 1$ implies
$$\lim_{k\rightarrow \infty}\left\|\sqrt{\frac{1-\alpha_{k1}^2}{1-\alpha_{ki}^2}}\sqrt{\frac{1-\alpha_{k1}^2}{1-\alpha_{kj}^2}}
P_{ij}^k\right\|\leq \rho_i\rho_j =0$$
for $p<i\leq j\leq r$, and for $i=0$.
It follows that
\begin{eqnarray*}
\lim_{k\rightarrow \infty} (1-\|z_k\|^2)B(z_k, z_k)^{-1/2}& = &
\lim_{k\rightarrow \infty} \sum_{0\leq i\leq j\leq r}
\sqrt{\frac{1-\alpha_{k1}^2}{1-\alpha_{ki}^2}}
\sqrt{\frac{1-\alpha_{k1}^2}{1-\alpha_{kj}^2}} \,P_{ij}^k\\
&= & \sum_{1\leq i\leq j\leq p}
\rho_i \rho_j P_{ij}(e_1, \ldots, e_p).
\end{eqnarray*}
Therefore we have
\begin{eqnarray*}
\lim_{k\rightarrow \infty} \frac{1-\|z_k\|^2}{1-\|g_{-x}(z_k)\|^2}  &=&
\lim_{k\rightarrow \infty} (1-\|z_k\|^2)\|B(x,x)^{-1/2}B(x,z_k)B(z_k,z_k)^{-1/2}\| \\
&= & \left\|\sum_{1\leq i\leq j\leq p}\rho_i\rho_jB(x,x)^{-1/2}B(x,\xi)P_{ij}(e_1, \ldots, e_p)\right\| \quad (x\in D).
\end{eqnarray*}
and hence $$F_\xi(x) = \left\|\sum_{1\leq i\leq j\leq p}\rho_i\rho_jB(x,x)^{-1/2}B(x,\xi)P_{ij}\right\|
\quad (P_{ij}=P_{ij}(e_1, \ldots, e_p)). $$
 \end{proof}

\begin{remark}\label{reverse}
Since $\|g_{-x}(z_k)\| =\|g_{-z_k}(x)\|$, the preceding  horofunction $F_\xi$ can also be
expressed as
\begin{eqnarray*}
F_\xi(x) &  =&
\lim_{k\rightarrow \infty} (1-\|z_k\|^2)\|B(z_k,z_k)^{-1/2}B(z_k,x)B(x,x)^{-1/2}\| \\
&= & \left\|\sum_{1\leq i\leq j\leq p}\rho_i\rho_jP_{ij}B(\xi, x)B(x,x)^{-1/2}\right\|.
\end{eqnarray*}
We will provide in (\ref{>0}) a sharper expression of $F_\xi$ in which the coefficients
$\rho_i$ are all positive.
\end{remark}


We note that the coefficients $\rho_i$ in the formula (\ref{fxi}) for $F_\xi$
depend on the sequence $(\alpha_{ki})$. Nevertheless,  we can also compute
$F_\xi$ via a special sequence $(y_k)$ given below.

\begin{corollary}\label{anotherseq} The function $F_\xi$ in Theorem \ref{horo}, defined by the sequence
 $$z_k = \alpha_{k1}e_{k1} + \cdots + \alpha_{kr}e_{kr}\qquad (\|z_k\|=\alpha_{k1} \geq \alpha_{k2} \geq \cdots
\geq \alpha_{kr} \geq 0),$$
with limit $\xi = \alpha_1 e_1 + \alpha_2 e_2 + \cdots \alpha_p e_p\,$,
can be computed by the limit
$$F_\xi (x) = \lim_{k\rightarrow \infty} \frac{1-\|y_k\|^2}{1-\|g_{-x}(y_k)\|^2}\qquad (x\in D)$$
where the sequence  $(y_k)$ is chosen to be
$$y_k = \alpha_{k1}e_1 + \alpha_{k2}e_2 + \cdots + \alpha_{kp}e_{p}.$$
\end{corollary}

\begin{proof}
Observe that the sequence $(y_k)$ norm converges to $\xi$ and
\begin{eqnarray*}
&&    \lim_{k\rightarrow \infty} (1-\|y_k\|^2)B(y_k, y_k)^{-1/2}\\ &= &
\lim_{k\rightarrow \infty} \sum_{0\leq i\leq j\leq p}
\sqrt{\frac{1-\alpha_{k1}^2}{1-\alpha_{ki}^2}}
\sqrt{\frac{1-\alpha_{k1}^2}{1-\alpha_{kj}^2}} \,P_{ij} \qquad (\alpha_{k0}=0)\\
&=&  \sum_{0\leq i\leq j\leq p}
\rho_i \rho_j P_{ij}
=  \sum_{1\leq i\leq j\leq p}
\rho_i \rho_j P_{ij}
\end{eqnarray*}
where $\rho_i = \lim_k \sqrt{\frac{1-\alpha_{k1}^2}{1-\alpha_{ki}^2}}$ for $i=0, 1,\ldots, p$ and
$\rho_0=0$.

Hence
\begin{eqnarray*}
&& \lim_{k\rightarrow \infty}  \frac{1- \|y_k\|^2}{1-\|g_{-x}(y_k)\|^2}
= \lim_{k\rightarrow \infty}   (1-\|y_k\|^2)\|B(x,x)^{-1/2}B(x,y_k)B(y_k,y_k)^{-1/2}\| \\
&=& \left\|\sum_{1\leq i\leq j\leq p}\rho_i\rho_jB(x,x)^{-1/2}B(x,\xi)P_{ij}\right\|
= F_\xi(x).
\end{eqnarray*}
\end{proof}

\begin{example}\label{xilambda} Let $F_\xi$ be the horofunction in Corollary \ref{anotherseq}, where
$\xi = \alpha_1e_1+ \cdots +\alpha_p e_{p}$ and
$$F_\xi (x) = \lim_{k\rightarrow \infty} \frac{1-\|y_k\|^2}{1-\|g_{-x}(y_k)\|^2}\qquad (x\in D)$$
with
$$y_k = \alpha_{k1}e_1 + \alpha_{k2}e_2 + \cdots + \alpha_{kp}e_{p}.$$
Then for each $t \in (0,1)$, we have
$$F_\xi(t \xi) = \frac{1-t}{1+t}.$$
Indeed, we have
$$g_{-t \xi}(y_k) =\psi_{-t\alpha_1}(\alpha_{k1}) e_1 + \cdots + \psi_{-t\alpha_p}(\alpha_{kp}) e_p$$
by Lemma \ref{inc}, and for $j=1, \ldots, p$, we have $\alpha_{kj} \geq t\geq t\alpha_j$ from some $k$ onwards
since $\alpha_{kj} \rightarrow 1$ as $k \rightarrow \infty$. Hence
$$\|g_{-t \xi}(y_k)\| = \max \{\psi_{-t\alpha_1}(\alpha_{k1}), \ldots,  \psi_{-t\alpha_p}(\alpha_{kp})\}
=\psi_{-t\alpha_1}(\alpha_{k1})= \psi_{-t}(\alpha_{k1})$$ and
\begin{eqnarray*}
 F_\xi({t \xi}) &=&\lim_{k\rightarrow\infty} \frac{1- \alpha_{k1}^2}{1-\|g_{-t \xi}(y_k)\|^2}
=  \lim_{k\rightarrow\infty} \frac{1- \alpha_{k1}^2}{1- |\psi_{-t}(\alpha_{k1})|^2}\\
&=& \lim_{ k \rightarrow \infty} \frac{(1- \alpha_{k1}^2)|1-t \alpha_{ k1}|^2}
{(1-t^2)(1-\alpha_{ k1}^2)}=\frac{1-t}{1+t}.
\end{eqnarray*}
\end{example}

For later application, we need a more explicit description of the horoballs $H(\xi,s)$.
 By \cite[Theorem 5.11]{horo}, the horoballs $H(\xi,s)$,
with defining horofunction $F_\xi$  in Theorem \ref{horo} and
\begin{equation}\label{47}
\xi = \lim_{k\rightarrow\infty} z_k = \alpha_1 e_1 + \alpha_2 e_2 + \cdots \alpha_p e_p \in \partial D,
\end{equation}
 are given by
\begin{equation}\label{hbr}
H(\xi, s)= \sum_{j=1}^p \frac{\sigma_j}{\sigma_j+s} e_j +
B\left (\sum_{j=1}^p \sqrt{\frac{\sigma_j}{\sigma_j+s}} e_j, \,\sum_{j=1}^p \sqrt{\frac{\sigma_j}{\sigma_j +s}} e_j
\right)^{1/2}(D)
\end{equation}
where $\sigma_j =\rho_j^2\in [0,1]$ \cite[(5.6)]{cr} and $\rho_j$ is given in (\ref{fxi})
for $j=1, \ldots,p$.
The second term on the right-hand side of (\ref{hbr}) is the Bergman operator.
In particular, $H(\xi,s) \neq \emptyset$ and
we note from (\ref{lambda}) that
$$\sigma_1 =1 \geq \sigma_2 \geq \cdots \geq \sigma_p\geq 0$$
and $\sigma_j=0$ if $\alpha_j <1$.

Evidently, the horoball $H(\xi,s)$ in (\ref{hbr})
only depends on the {\it positive} coefficients $\sigma_j$.
Let
$$q= \max \{j\in \{1, \ldots, p\}: \sigma_j >0\}.$$
Note that $q \geq 1$ and $\sigma_j >0$ implies $\alpha_j =1$ by (\ref{lambda}).

By Corollary \ref{anotherseq}, $F_\xi$ can be computed by the sequence
$$y_k=\alpha_{k1} e_1 + \cdots + \alpha_{kp}e_p.$$
Let $c= e_1 + \cdots +e_q$ and
let $F_c$ be the horofunction
defined by the truncated sequence
\begin{equation}\label{trunc}
z_k'=\alpha_{k1} e_1 + \cdots + \alpha_{kq}e_q
\end{equation}
which converges to the tripotent $c$.
Then the horoballs
$$H(c, s)= \{x\in D: F_c(x)<s\} \qquad (s>0)$$
 defined by the sequence
$( z_k')$ is identical with $H(\xi,s)$.

Indeed, by  \cite[Theorem 5.11]{horo} again, we have,
with $\sigma_j = \lim_k \frac{1- \alpha_{k1}^2}{1-\alpha_{kj}^2}$ for $j=1, \ldots, q$,
 \begin{eqnarray}\label{e=xi}
H(c,s) &=&  \sum_{j=1}^q \frac{\sigma_j}{\sigma_j+s} e_j +
B\left (\sum_{j=1}^q \sqrt{\frac{\sigma_j}{\sigma_j+s}} e_j, \,\sum_{j=1}^q \sqrt{\frac{\sigma_j}{\sigma_j +s}} e_j
\right)^{1/2}(D)\nonumber\\
&=& H(\xi,s)
\end{eqnarray}
 for all $s>0$  since $\sigma_j =0$ for $j= q+1, \ldots, p$ in (\ref{hbr}). It follows that
\begin{equation}\label{>0}
 F_\xi (x) =F_c (x)  = \left\|\sum_{1\leq i\leq j\leq q}\rho_i\rho_jB(x,x)^{-1/2}B(x,c)P_{ij}\right\|
\end{equation}
where  $\rho_j = \sqrt\sigma_j>0$ for $j=1, \ldots,q$ and
$P_{ij} = P_{ij}(e_1, \ldots, e_q)$ is the Peirce projection of the minimal tripotents $e_1, \ldots, e_q$
satisfying
\begin{equation*}
P_{ij}(e_1, \ldots, e_q) = P_{ij}(e_1, \ldots, e_q, \ldots, e_p) \qquad (1\leq i\leq j\leq q).
\end{equation*}
by (\ref{lemma 2.1}).

\begin{definition}
In (\ref{e=xi}), we denote
\begin{equation}\label{center}
c_s =  \frac{1}{1+s} e_1 +  \frac{\sigma_2}{\sigma_2+ s} e_2 + \cdots +  \frac{\sigma_q}{\sigma_q+s} e_q
\in D
\end{equation}
and call it the {\it centre} of the horoball $H(c,s)=H(\xi,s)$. The tripotent
$$c=e_1 + \cdots + e_q = \lim_{s\rightarrow 0} c_s \in \partial D$$ is called
the {\it horocentre} of $H(c,s)$. We call $s$ the {\it hororadius} of $H(c,s)$.
\end{definition}

We can view  in some way the horoball $H(\xi,s)=H(c,s)$
 as some kind of  {\it `open ball'} contained in D. Indeed, it {\it is}
a proper open disc in the case $D=\mathbb{D}$.
To see this, let us denote for short
\begin{equation}\label{bberg}
B_s= B\left (\sum_{j=1}^q \sqrt{\frac{\sigma_j}{\sigma_j+s}} e_j, \,\sum_{j=1}^q \sqrt{\frac{\sigma_j}{\sigma_j +s}} e_j
\right).
\end{equation}

\begin{proposition} \label{hbrs}
Let $s>0$ and $H(\xi, s)$ be the horoball in a finite-rank domain $D$ defined by a sequence $(z_k)$
converging to the boundary point $\xi$, with
spectral decomposition $$\xi = \alpha_1 e_1 + \alpha_2 e_2 + \cdots + \alpha_p e_p \qquad (\alpha_1=1).$$
Let $c_s$ be the centre of $H(\xi,s)$ and for $\rho>0$, let
$S\left(c_s, \rho\right)= \left\{x\in D: \|x-c_s\| < \rho \right\}$
be the open ball with center $c_s$ and radius $\rho$.
 Then we have
$$ S\left(c_s, \frac{s}{1+s}\right) \subset H(\xi,s) = \{x\in D: \|B_s^{-1/2}(x-c_s)\| <1\}\subset S(c_s, \|B_s^{1/2}\|).$$
If $D=\mathbb{D}$, then
\begin{equation}\label{unitdisc}
H(\xi,s) = S\left(c_s, \frac{s}{1+s}\right) = \frac{1}{1+s} \xi + \frac{s}{1+s}\mathbb{D}.
\end{equation}
\end{proposition}

\begin{proof}
By (\ref{e=xi}), each $x\in H(\xi,s)$ is of the form $x= c_s + B_s^{1/2}(y)$ for some $y\in D$.
Hence $\|B_s^{-1/2}(x-c_s)\| = \|y\| <1$ and $\|x-c_s\|\leq \|B^{1/2}_s(y)\| <\|B_s^{1/2}\|$, that is,
$x \in  S(c_s, \|B_s^{1/2}\|).$

Conversely, given $\|B_s^{-1/2}(x-c_s)\| <1$, then we have
$$x = c_s + B_s^{1/2}(B_s^{-1/2}(x-c_s)) \in H(\xi,s).$$

For each $x\in S\left(c_s, \frac{s}{1+s}\right)$, we have
$\|x-c_s\| <  \frac{s}{1+s}$ and
$$\|B_s^{-1/2}(x-c_s)\|\leq \|B_s^{-1/2}\| \|x-c_s\| < \frac{1}{1-\left(\sqrt{\frac{1}{1+s}}\right)^2}\left(\frac{s}{1+s}\right) =1$$
by (\ref{bid2}) and (\ref{osum}). Hence $x\in H(\xi,s)$.

If $D=\mathbb{D}$, then $\xi$ reduces to $e_1\in \partial \mathbb{D}$ and $c_s = \frac{1}{1+s}e_1$.
By (\ref{bz}),
$$B_s^{1/2}(y) = B\left(\sqrt{\frac{1}{1+s}} e_1, \sqrt{\frac{1}{1+s}} e_1\right)^{1/2}(y) = \frac{s}{1+s}y$$
which yields (\ref{unitdisc}).

\end{proof}

 We note that, for each $w\in \overline D$, we can write
 \begin{eqnarray}\label{419}
&& c_s + B_s^{1/2}(w) =\frac{1}{1+s} e_1 + \cdots +  \frac{\sigma_q}{\sigma_q +s}e_q + \sum_{0\leq i\leq j \leq q}
\sqrt{\frac{s}{\sigma_i +s}}\sqrt{\frac{s}{\sigma_j +s}}P_{ij}(w)\nonumber\\
 &=&\frac{1+st_1}{1+s} e_1 + \cdots +  \frac{\sigma_q+st_q}{\sigma_q +s}e_q + \sum_{0\leq i<j \leq q}
\sqrt{\frac{s}{\sigma_i +s}}\sqrt{\frac{s}{\sigma_j +s}}P_{ij}(w)
\end{eqnarray}
 where $P_{ii}(w) = P_2(e_i)(w) = t_i w$  and $|t_i|\leq 1$ for $i= 1, \ldots, q$.

\begin{example}\label{2disc}
In the special case of the unit disc $\mathbb{D}$, we call the horoball in (\ref{unitdisc}) a {\it horodisc} and denote
it by $\mathbb{H}(\xi,s)=H(e_1,s)$,
of which the horocentre is $e_1\in \partial \mathbb{D}$.
A minimal tripotent $e$ in the bidisc $\mathbb{D}^2$ is of the form $e=(e^{i\alpha},0)$ or
$e=(0, e^{i\beta})$ for some $\alpha, \beta \in \mathbb{R}$. Given two orthogonal minimal tripotents say,
$e_1=(e^{i\alpha},0)$ and $e_2=(0, e^{i\beta})$, a horoball $H(e_1+e_2, s)$
in (\ref{e=xi}) defined by a sequence in $\mathbb{D}^2$ is a product of two horodiscs:
$$H(e_1+e_2, s) = \frac{1}{1+s}e_1 + \frac{\sigma}{\sigma +s}e_2 + B_s^{1/2}(\mathbb{D}^2)=
\mathbb{H}(e^{i\alpha},s) \times \mathbb{H}(e^{i\beta}, s/\sigma) \quad (\sigma >0).$$
 This can be seen as follows using the coordinatewise Jordan triple
product in $\mathbb{C}^2$ defined in Example \ref{ellinfty}, with  (\ref{bz}).
\begin{eqnarray*}
B_s^{1/2}(z,w) &=& B\left(\left (\sqrt{\frac{1}{1+s}}e^{i\alpha} , \sqrt{\frac{\sigma}{\sigma+s}} e^{i\beta}\right),
\left(\sqrt{\frac{1}{1+s}}e^{i\alpha} , \sqrt{\frac{\sigma }{\sigma+s}}e^{i\beta}\right)\right)^{1/2}(z,w)\\
 &=& \left(\left(1-\left|\sqrt{\frac{1}{1+s}} e^{i\alpha}\right|^2\right) z,\,
 \left(1-\left|\sqrt{\frac{\sigma}{\sigma+s}} e^{i\beta}\right|^2\right)w\right)\\
& =&\left( \frac{s}{1+s} z,  \frac{s/\sigma}{1+s/\sigma}w\right) \qquad (z,w\in \mathbb{D}^2).
\end{eqnarray*}
\end{example}

\begin{xrem} Generalisations of the horodisc (\ref{unitdisc}) in $\mathbb{D}$
 to higher dimensions have been given by a number of authors, often called {\it horospheres}
(e.g. \cite{aba,ab,yang}).
\end{xrem}

We see from (\ref{hbr}) that the closure $\overline H(\xi,s)$ is give by
\begin{equation}\label{clos}
\overline H (\xi,s) = c_s + B_s^{1/2}(\overline D) = \{x\in \overline D : \|B_s^{-1/2}(x-c_s)\|\leq 1\}.
\end{equation}
In particular, $\xi\in \overline H(\xi,s)$ for all $s>0$ since
\begin{eqnarray*}
\|B_s^{-1/2}( \xi-c_s) \|&=&  \left\|\sum_{j=1}^p  \left(1- \left(\sqrt{\frac{\sigma_j}{\sigma_j+s}}\right)^2\right)^{-1}
 \left(\alpha_j - \frac{\sigma_j}{\sigma_j+s}\right)e_j \right\|\\
 &=& \left\|\sum_{j=1}^p\left(\frac{\sigma_j +s}{s}\right)  \left(\alpha_j - \frac{\sigma_j}{\sigma_j+s}\right)e_j\right\|
 =\|e_1 + \cdots\|= 1.
 \end{eqnarray*}

\begin{remark}\label{513}It has been shown in \cite[Theorem 5.13]{horo} that
$$\overline H(\xi,s) \cap D = \{x\in D: F_\xi(s) \leq s\}.$$
We will call $\overline H(\xi,s)$  a {\it closed horoball}.

\end{remark}
We now show that the intersection of all closed horoballs in $D$ forms
the closure of a holomorphic boundary component
of $D$, which is the final key to the Denjoy-Wolff theorem.

\begin{theorem}\label{cr1}
Let $D$ be a bounded symmetric domain of rank $r< \infty$, realised as the open
unit ball of a JB*-triple $V$. Let $(z_k)$ be a sequence in $D$ norm convergent to $\xi\in \partial D$,
which defines the horoballs
$$H(\xi, s) = \{x\in D: F_\xi(x)<s\}  \qquad (s>0)$$
with horofunction
$F_\xi$.
Then the intersection $\bigcap_{s>0}\overline H (\xi,s)$
is the closure of a holomorphic boundary component of $D$. More explicitly,
$$\bigcap_{s>0}\overline H (\xi,s)= c + (V_0(c) \cap \overline D) = \overline \Gamma_c$$
where $c\in \partial D$ is the horocentre of $H(\xi,s)$  and $V_0(c)$ is the Peirce $0$-space of $c$.
\end{theorem}

\begin{proof} Let $\xi= \alpha_1 e_1 + \cdots + \alpha_p e_p$ ,  with $\alpha_p >0$
and $p\leq r$, as in Theorem \ref{horo}. By  (\ref{e=xi}) and (\ref{clos}), we have
$$\overline H(\xi,s)=\sum_{j=1}^q \frac{\sigma_j}{\sigma_j +s} e_j +
B\left (\sum_{j=1}^q \sqrt{\frac{\sigma_j}{\sigma_j+s}} e_j, \,\sum_{j=1}^q \sqrt{\frac{\sigma_j}{\sigma_j+s}} e_j
\right)^{1/2}(\overline D)$$
where $1=\sigma_1 \geq \cdots \geq \sigma_q > 0$.
Let
 \begin{equation}\label{sume}
 c= e_1 + \cdots +e_q = \lim_{s\rightarrow 0}\sum_{j=1}^q \frac{\sigma_j}{\sigma_j+s} e_j
= \lim_{s\rightarrow 0}\sum_{j=1}^q \sqrt{\frac{\sigma_j}{s+\sigma_j}}\, e_j
\end{equation}
which is the horocentre of $H(\xi,s)$.

We show $$\bigcap_{s>0}\overline H (\xi,s) = \overline \Gamma_c.$$

Let $x\in \bigcap_{s>0}\overline H (\xi,s)$. Then for $n=1,2,\ldots,$
we have $x \in \overline H(\xi,1/n)$ and
$$ x= \sum_{j=1}^q \frac{\sigma_j}{\sigma_j + 1/n} e_j +
B\left (\sum_{j=1}^q \sqrt{\frac{\sigma_j}{\sigma_j + 1/n}} e_j, \,\sum_{j=1}^q \sqrt{\frac{\sigma_j}{\sigma_j + 1/n}} e_j
\right)^{1/2}(x_n)$$
for some $x_n \in \overline D$. By relative weak compactness of $D$, there is a subsequence $(x_k)$ of $(x_n)$
weakly converging to some point $w\in \overline D$. It follows that, letting $k\rightarrow \infty$,
$$x = c\, + \, weak\mbox{-}\lim_{k\rightarrow \infty}\,
B\hspace{-.03in}\left(\sum_{j=1}^q \sqrt{\frac{\sigma_j}{\sigma_j + 1/k}} e_j, \,\sum_{j=1}^q \sqrt{\frac{\sigma_j}{\sigma_j + 1/k}} e_j
\right)^{1/2}(x_k)$$
by (\ref{sume}), where the sequence of Bergman operators is norm convergent,  with limit
 $$\lim_{k\rightarrow \infty}\,
B\hspace{-.03in}\left(\sum_{j=1}^q \sqrt{\frac{\sigma_j}{\sigma_j + 1/k}} e_j, \,\sum_{j=1}^q \sqrt{\frac{\sigma_j}{\sigma_j + 1/k}} e_j
\right)^{1/2} =B(c,c)^{1/2}=P_0(c).$$
 Hence $x= c+ P_0(c)w \in c+ (V_0(c)\cap \overline D) = \overline \Gamma_c$.

 Conversely, let $y\in \overline \Gamma_c$. Then there is a point $z\in \overline D$
 such that
$$y= c+ P_0(c)(z) = c+ B(c,c)^{1/2}(z).$$
For $n=1, 2, \ldots$, let
$$y_n =   \sum_{j=1}^q \frac{\sigma_j}{\sigma_j + 1/n} e_j  + B\hspace{-.03in}\left(\sum_{j=1}^q \sqrt{\frac{\sigma_j}{\sigma_j + 1/n}} e_j, \,\sum_{j=1}^q \sqrt{\frac{\sigma_j}{\sigma_j + 1/n}} e_j
\right)^{1/2}(z) \in \overline H (\xi,1/n). $$
Then we have $ \lim_n y_n =y$.

If $y \notin \bigcap_{s>0}\overline H (\xi,s)$, then $y \notin\overline H (\xi,s_0)$ for some $s_0 > 0$.
Hence there exists $N \in \mathbb{N}$ such that $y_n \notin\overline H (\xi,s_0)$ for all $n >N$.
Pick $n'>N$ such that $1/n' < s_0$. Then $y_{n'} \notin\overline H (\xi,s_0)$ whereas
$y_{n'} \in\overline H (\xi, 1/n') \subset \overline H(\xi,s_0)$, which is impossible. Therefore $y \in
\bigcap_{s>0}\overline H (\xi, s)$ and the proof is complete.
\end{proof}

\begin{remark}\label{cr2}
Given a sequence $(s_n)$ in $(0,1)$ decreasing to $0$, we see readily, following the preceding arguments, that
$\bigcap_n \overline H(\xi,s_n) = \overline \Gamma_c$.
More generally, given $(s_n)$ decreases to some $s>0$, one can show analogously that
$\bigcap_n \overline H(\xi,s_n) = \overline H(\xi, s)$.  Also, given a sequence $w_n \in \overline H(\xi, s_n)$
converging to $w \in \overline D$, we have $w\in \bigcap_n \overline H(\xi,s_n)$.
\end{remark}

In the case of Hilbert balls, we have a simpler result.

\begin{theorem}\label{1ptintersection} Let  $D$ be a  Hilbert ball. Then each horoball $H(\xi,s)$ in Theorem \ref{cr1}
  satisfies
 $$\overline{H}(\xi, s) \cap \partial D
=\{\xi\}.$$
\end{theorem}
\begin{proof}
 By a remark before Theorem \ref{cr1}, we have $\xi \in \overline{H}(\xi,s)$.

By Example \ref{h}, we have
$$F_\xi(x)= \frac{|1-\langle x, \xi\rangle|^2}{1-\|x\|^2} \qquad (x\in D)$$
where $\langle \cdot, \cdot\rangle$ denotes the underlying inner product.

Let $x\in \overline{H}(\xi,s)  \cap \partial  D$. Then $\|x\|=1$ and $x=\lim_n x_n$ for some sequence
$(x_n)$ in $H(\xi,s)$, which  satisfies $F_\xi(x_n)<s $. It follows that
$$|1-\langle x,\xi\rangle|^2=\lim_{n\rightarrow \infty} |1-\langle x_n,\xi\rangle|^2
= \lim_{n\rightarrow \infty}F_\xi(x_n) (1-\|x_n\|^2) \leq \lim_{n\rightarrow \infty}s(1-\|x_n\|^2) = 0$$
 yielding $x=\xi$, which completes the proof.
\end{proof}

\section{Denjoy-Wolff Theorem}

In this final section, we prove a Denjoy-Wolff theorem for a fixed-point free compact holomorphic map $f: D \longrightarrow D$
on a bounded symmetric domain $D$ of finite rank, including the case of all fixed-point free
holomorphic self-maps on finite dimensional bounded symmetric domains. As before, we identify $D$ with the open unit ball
of a JB*-triple.

In the language of dynamics,
which studies the limit set
$\omega(a)$ of each point $a\in D$,
consisting of limit points of  the orbit $(f^n(a))$:
 $$\omega(a)=\bigcap_{n\in \mathbb{N}} \overline{\bigcup_{k\geq n}\{f^k(a)\}},$$
the Denjoy-Wolff theorem
reveals a striking phenomenon that {\it all} limit sets $\bigcup_{a\in D} \omega(a)$
accumulate in one single face $\overline\Gamma$
of $\overline D$.

Indeed, given a limit point $\xi= \lim_k f^{n_k}(a)$ in $\omega(a)$, we can find a locally uniformly convergent
subsequence $(f^{m_k})$ of $(f^{n_k})$ by Lemma \ref{cp} below. Hence $\xi=\ell(a)$, where
$\ell$ is the subsequential limit  $\ell = \lim_k f^{m_k}$ and the Denjoy-Wolff theorem says $\ell(D) \subset \overline\Gamma$.

We begin with some lemmas.

\begin{lemma}\label{cp} Let $f: D \longrightarrow D$ be a compact holomorphic map. Then every subsequence
of the iterates $(f^n)$ contains a locally uniformly convergent subsequence.
\end{lemma}
\begin{proof}
This has been proved in \cite[Lemma 1, Remark 1]{cm}.
\end{proof}

In contrast, we note  that a sequence of holomorphic self-maps on an infinite dimensional domain, even on a
separable  Hilbert ball,  need not
have any locally uniformly convergent subsequence \cite[Example 3.1]{kim}.

\begin{definition} A locally uniform subsequential limit  $\ell = \lim_k f^{n_k}$ of the iterates $(f^n)$
is called a {\it limit function} of  $(f^n)$.
\end{definition}

If a compact holomorphic map $f: D \longrightarrow D$ on a bounded symmetric domain $D$ has no fixed-point,
then the image $\ell(D)$ of each limit function $\ell = \lim_k f^{n_k}$
must lie entirely in the boundary $\partial D$. This is a consequence of
the following result proved in \cite{kk2}.

\begin{lemma}\label{kry}
Let $D$ be the open unit ball of a complex Banach space and $f: D\longrightarrow D$
a compact holomorphic map. Then $f$ is fixed-point free if and only if
$\sup_k \|f^{n_k}(a)\|=1$ for each subsequence $(f^{n_k})$ of the iterates $(f^n)$ and  each $a\in D$.
\end{lemma}

Given a fixed-point free compact holomorphic self-map $f$ on a bounded symmetric domain $D$, we will make use of
 the existence of a family of $f$-invariant horoballs $H(\xi,s)$ in $D$, constructed in \cite[Theorem 4.4]{horo}
 and stated below. This fact is a generalisation of  a boundary version of the Schwarz lemma for $\mathbb{D}$, proved by
 Wolff \cite{w1}.

\begin{lemma} \label{g} Let $f$ be a fixed-point free compact holomorphic self-map on a bounded symmetric domain
$D$ of any rank in $\mathbb{N}\cup \{\infty\}$.
Then there is a sequence $(z_k)$ in $D$ converging
to a boundary point $\zeta \in \partial D$ such that the horoballs
$H(\zeta, s)$ defined by $(z_k)$ are $f$-invariant, that is, $f(H(\zeta,s)) \subset H(\zeta,s)$,
 for all $s>0$.

Further, if $D$ is a Hilbert ball, the compactness assumption on $f$ can be removed.
\end{lemma}
\begin{proof} For completeness,
we show the construction of $(z_k)$ and refer to  \cite[Theorem 4.4]{horo}
for other details.

Choose an increasing sequence $(\beta_k)$ in $(0,1)$ with
limit $1$. Then $\beta_k f$ maps $D$ strictly inside itself and
 the fixed-point theorem of Earle and Hamilton \cite{eh} implies the existence of some $z_k \in D$
such that $\beta_kf(z_k) = z_k$. Note that $z_k \neq
0$. Since $\overline{f(D)}$ is compact, we may assume, by choosing
a subsequence if necessary, that $(z_k)$ converges to a point $\zeta
\in \overline D$. Since $f$ has no fixed point in $D$, the point
$\zeta$ must lie in the boundary $\partial D$. Choosing a subsequence, we may assume that
$(z_k)$ satisfies (\ref{sigi}).

The $f$-invariance of  the horoballs $H(\zeta,s)$ defined by $(z_k)$
is a consequence of $f= \lim_k \beta_k f$ and the Schwarz-Pick lemma for $D$ \cite[Lemma 3.5.18]{book2}:
$$ \|g_{-z_k}(\beta_k f(x))\| =\|g_{-\beta_k f (z_k)}(\beta_k f (x)) \leq \|g_{-z_k}(x)\|
\qquad (x\in H(\zeta, s)).$$

For a Hilbert ball $D$, the compactness assumption of $\overline{f(D)}$ can be dropped.
Indeed, by weak compactness of $\overline D$ and choosing a subsequence in this case, we may assume
that $(z_k)$ converges {\it weakly} to some $\xi \in \overline D$. We show $\|\xi\|=1$.
Otherwise $\xi \in D$
and the Schwarz-Pick lemma implies $ \|g_{-f( \xi)}(f( z_k))\| \leq  \|g_{- \xi}( z_k)\|$.

Using (\ref{hh}) and substituting $ f( z_k) =\beta_k^{-1} z_k$, one deduces
$$\frac{|1-\langle \beta_k^{-1} z_k, f( \xi)\rangle|^2}{1-\|f( \xi)\|^2}\leq
\frac{1-\beta_k^{-2}\| z_k\|^2}{1-\| z_k\|^2} \frac{|1-\langle  z_k,
 \xi\rangle|^2}{1-\| \xi\|^2}.$$
Since $\frac{1-\beta_k^{-2}\| z_k\|^2}{1-\| z_k\|^2} \leq 1$, letting $k \rightarrow \infty$ gives
$$\frac{1}{1-\|g_{-f( \xi)}( \xi)\|^2} \leq 1$$
and hence $\|g_{-f( \xi)}( \xi)\|=0$. This implies
 $f( \xi) =  \xi$, contradicting the  non-existence of a fixed-point in $D$. It follows that
  $( z_{k})$ actually {\it norm} converges to $ \xi$
   since $\limsup_{k\rightarrow \infty} \| z_{k}\| \leq 1
  =\| \xi\| \leq \liminf_{k\rightarrow \infty}\| z_{k}\|$.
The horoballs $H(\xi,s)$ defined by $(z_k)$ are $f$-invariant.
\end{proof}

 Finally we are ready to prove the Denjoy-Wolff theorem.

 \begin{theorem}\label{dw}
Let $D$ be a bounded symmetric domain of finite rank, realised as the open unit ball of a complex Banach space,  and  $f: D
  \longrightarrow D$  a fixed-point free compact holomorphic map such that
  the extended Shilov boundary $\Sigma^*(D)$ contains the limit points of one orbit $\{a, f(a), f^2(a), \ldots \}$
  in an $f$-invariant horoball of unit hororadius.
Then there is a holomorphic boundary component $\Gamma$ in the boundary $\partial D$ such that
$\ell(D) \subset \overline \Gamma$ for all limit functions
$\displaystyle \ell= \lim_{m\rightarrow \infty} f^{m}$ of the iterates
$(f^n)$.

If $D$ is the bidisc or a Hilbert ball, the condition on $\Sigma^*(D)$ is superfluous.
\end{theorem}

\begin{proof}
We first remark that in the case of a Hilbert ball $D$, the manifold $\Sigma^*(D)$ coincides with the boundary $\partial D$
and contains all limit points of all orbits by Lemma \ref{kry}.

For the bidisc $\mathbb{D}^2$, Herv\'e has shown the result in \cite{h1} without the condition on
$\Sigma^*(D)$. We will present a proof simplifying Herv\'e's in the Appendix.

Let $D$ be of rank $r$, realised  as the open unit ball of a JB*-triple $V$.
By Lemma \ref{g}, $D$ contains an abundance of $f$-invariant horoballs.

Let $H(\zeta,1)$ be  the $f$-invariant horoball in the condition on $\Sigma^*(D)$.
It is defined by an $f$-invariant horofunction $F_\zeta$ and a sequence $(z_k)$ in $D$, with spectral decomposition
$$z_k= \beta_{k1}c_{k1} + \cdots +\beta_{kr}c_{kr} \quad (\|z_k\|=\beta_{k1} \geq \cdots \geq \beta_{kr} \geq 0)$$
which converges to $\zeta\in \partial D$.

The $f$-invariant horoballs $H(\zeta,s)=\{x\in D:F_\zeta(x)<s\}$, for $s>0$, are of the form
\begin{equation}\label{hball}
H(\zeta,s) = \sum_{j=1}^d \frac{\sigma_j}{\sigma_j+s} c_j +
B\left (\sum_{j=1}^d \sqrt{\frac{\sigma_j}{\sigma_j+s}} c_j, \,\sum_{j=1}^d \sqrt{\frac{\sigma_j}{\sigma_j +s}} c_j
\right)^{1/2}(D)
\end{equation}
with $1 =\sigma_1 \geq \sigma_2 \geq \cdots \geq \sigma_d >0$
and $\zeta$ admits a spectral decomposition
$$\zeta = c_1 + \cdots + c_d + \beta_{d+1} c_{d+1} + \cdots + \beta_r c_r \qquad (
 \beta_{d+1} \geq \cdots \geq \beta_r \geq 0)$$
 where  $1\leq d \leq r$ and
$$c= c_1 + \cdots +c_d$$
is the horocentre of $H(\zeta,s)=H(c,s) = \{x\in D: F_c(x) <s\}$ by (\ref{e=xi}).
The horofunction $F_c=F_\zeta$ is given by
 \begin{eqnarray}\label{Fc}
 F_c(x) 
 &=& \left\|\sum_{1\leq i\leq j\leq d}\rho_i\rho_jP_{ij}(c_1, \ldots, c_d)B(c,x)B(x,x)^{-1/2}\right\|
 \end{eqnarray}
(cf. Remark \ref{reverse}, (\ref{>0})) and
$\rho_j^2 = \sigma_j>0$ for $j=1, \ldots, d$. By (\ref{trunc}), $F_c$ can be computed by the sequence
$$z_k' =  \beta_{k1}c_{1} + \cdots +\beta_{kd}c_{d}.$$
By Theorem \ref{cr1}, the intersection
$$ \bigcap_{s>0} \overline H(c,s) =\overline\Gamma_c $$
 is the closure of the holomorphic boundary component $\Gamma_c = c+ (V_0(c)\cap D) \subset \partial D$.

The given condition on $\Sigma^*(D)$ provides an orbit $\{a_0, f(a_0), \ldots, \}\subset H(c,1)$ of which the limit points
are  maximal or structural  tripotents. In particular, $F_c(a_0) <1$ implies $F_c(a_0) < s_0<1$ for some $s_0>0$, that is, $a_0 \in H(c,s_0)$.  Let
$$d_0 = \max\{i\in \{1, \ldots, d\}: s_0< \sigma_i\}.$$
Since $0<s_0<1$, we have $d_0 \geq 1$. Also,  $t\in \overline{\mathbb{D}}$
with $s_0 < \sigma_i$
implies $\frac{\sigma_i+s_{_0}t}{\sigma_i+s_{_0}}\neq 0$ for all $i=1, \ldots, d_0$.

Let $c_0 = c_1 + \cdots + c_{d_0}$ and let
$$\ell=\lim_{m\rightarrow \infty} f^{m}$$ be a limit function,
 where  $(f^{m})$ is a  subsequence of the iterates $(f^n)$.
We show that
 $$\ell(D) \subset \overline\Gamma_{c_{_0}} =c_0 + (V_0(c_0) \cap \overline D)  $$
which would complete the proof.

Let $a\in D$ and let
\begin{equation}\label{ite}
f^m(a) = \alpha_{m1}e_{m1} + \cdots +\alpha_{mr}e_{mr}\quad (1>\|f^m(a)\|=\alpha_{m1} \geq \cdots \geq \alpha_{mr} \geq 0).
\end{equation}
be a spectral decomposition of $f^m(a)$. We may assume, by choosing subsequences as in Remark \ref{sub},
that for $i=1, \ldots, r$, each sequence $(\alpha_{mi})_m$ converges to some $\alpha_i \in [0, 1]$,
and $(e_{mi})_m$ weakly converges $e_i  \in \overline D$.

By Lemma \ref{zmm},
there exists $r_1\in \{1, \ldots, r\}$ such that
\begin{enumerate}
\item[(i)] $\alpha_i >0$ and $(e_{mi})_m$ norm converges to $e_i$ for $1\leq i\leq r_1$,
\item[(ii)] $\alpha_{i} =0$ for $i = r_1+1, \ldots, r$.
\item[(iii)] $e_1, \ldots, e_{r_1}$ are mutually orthogonal minimal tripotents.
\end{enumerate}
Hence we have the spectral decomposition
$$ \ell(a)= \lim_m f^m(a) = \alpha_1e_1 +\cdots +\alpha_{r_1} e_{r_1}$$
where $\alpha_1 = \lim_m \alpha_{m1} =\lim_{m} \|f^m(a)\| = \|\ell(a)\|=1$.
Let
$$p= \max \{j\in \{1, \ldots, r_1\}: \alpha_j=1\}.$$
Then $1\leq p\leq r_1$ and  we can write
\begin{equation*}\label{axi}
\ell(a) = e_1 + \cdots + e_p + \alpha_{p+1}e_{p+1} + \cdots + \alpha_{r_1}e_{r_1} \qquad (1> \alpha_{p+1} \geq \cdots
\geq \alpha_{r_1}>0)
\end{equation*}
where $\alpha_1 = \cdots = \alpha_p =1$ and $e_1 + \cdots + e_p$ is a tripotent.

Write $\xi = e_1 + \cdots +e_p$. Then $\ell (a) \in \xi+(V_0(\xi) \cap D) $ as $e_{p+1}, \ldots, e_{r_1} \in V_0(\xi)$.
 By Lemma \ref{3.2}, we have
\begin{equation}\label{3..2}
\ell(D) \subset \Gamma_{\ell(a)} = \Gamma_\xi= e_1 + \cdots + e_p + (V_0(\xi) \cap D).
\end{equation}

By assumption, the limit point $\ell(a_0)= \lim_m f^m(a_0)$ is a maximal or structural tripotent and
$\ell(a_0) \in \Gamma_\xi$ implies $\ell(a_0) = \xi$ by Remark \ref{uni2}.

If $\xi$ is a maximal tripotent, then we have $\ell(D) \subset \Gamma_\xi = \{\xi\}$.
 For any $s>0$, pick $a'\in H(c,s)$, then the $f$-invariance of $H(c,s)$ implies
$$\xi = \ell(a') = \lim_{m\rightarrow \infty} f^m(a') \in \overline H(c,s)$$
resulting in $$\xi \in \bigcap_{s>0} \overline H(c,s) = \overline\Gamma_c \subset \overline\Gamma_{c_{_0}}$$
where $c= c_1 + \cdots + c_{d_0} + \cdots + c_d =c_0 + (c-c_0) \in c_0 + (V_0(c_0)\cap \overline D)$.
This proves  $\ell(D)=\{\xi\} \subset  \overline\Gamma_{c_{_0}}$.

Now let $\xi$ be a structural tripotent in the remainder of the proof.
By $f$-invariance, $a_0\in H(c, s_0)$ implies
$\xi= \lim_m f^m(a_0) \in \overline H(c,s_0)$.
Hence we can write, by (\ref{419}),
\begin{eqnarray}\label{j0}
&& \indent\hspace{-.8in}e_1 + \cdots + e_{p} = \xi = c_{s_0} + B_{s_0}^{1/2}(z) \nonumber\\
 &&\indent\hspace{-1in}= \frac{1+s_0t_1}{1+s_0} c_1 + \cdots +  \frac{\sigma_d+s_0t_d}{\sigma_d +s_0}c_d + \sum_{0\leq i<j \leq d}
\sqrt{\frac{s_0}{\sigma_i +s_0}}\sqrt{\frac{s_0}{\sigma_j +s_0}}P_{ij}(z)
\end{eqnarray}
for some $z\in \overline D$, where $P_{ii}(z) = P_2(c_i)(z) = t_i c_i$  and $|t_i|\leq 1$ for $i= 1, \ldots, d$.

Since $\frac{\sigma_i+s_{_0}t_i}{\sigma_i+s_{_0}}\neq 0$ for $i=1, \ldots, d_0$,
 we have,  in view of (\ref{j0}), that
$$P_2(c_i)(\xi) = \frac{\sigma_i+s_0t_i}{\sigma_i+s_0} c_i \neq 0$$
for each $i = 1, \ldots, d_0$. In what follows, we show $ c_i + \cdots +c_{d_0} \in V_2(\xi)$.

Let $s_n=1/n$ for $n=1,2, \ldots$. Pick $a_n \in H(c, s_n)$.
The $f$-invariance of $H(c, s_n)$ implies $\ell(a_n) \in \overline H(c,s_n)$, with a spectral decomposition
$$\ell(a_n) = e_1 + \cdots + e_p + \alpha_{n,p+1}e_{n, p+1} + \cdots + \alpha_{n, r}e_{n, r} \in \Gamma_\xi
~
(1> \alpha_{n, p+1} \geq \cdots \geq \alpha_{n,r}\geq 0)$$
by (\ref{3..2}), where $e_{n,p+1}, \ldots, e_{n,r} \in V_0(\xi)\cap \overline D$.
Since $\ell(a_n) =\lim_m f^m(a_n) \in \overline{f(D)}$ for $n=1,2, \ldots$
and $\overline{f(D)}$ is compact, choosing a subsequence, we may assume that
the limit
\begin{equation}\label{w}
 w=\lim_{n\rightarrow \infty} \ell(a_n)\in \overline\Gamma_\xi = \xi +(V_0(\xi)\cap \overline D)
 \end{equation} exists.
On the other hand,  Remark \ref{cr2} implies
\begin{equation}\label{w'}
 w \in \bigcap_{n=1}^\infty \overline H(c, s_n) = \overline \Gamma_c= c + (V_0(c) \cap \overline D)
 \end{equation}
since $\displaystyle \lim_{n\rightarrow \infty} s_n =0$.

In  view of (\ref{w}) and (\ref{w'}),  $w$ has two spectral decompositions
\begin{eqnarray}\label{w11}
w &=& e_1 + \cdots + e_p + \alpha_{ p+1}e_{ p+1} + \cdots + \alpha_{r}e_{r}
\quad (1 \geq \alpha_{p+1} \geq \cdots \geq \alpha_r \geq 0)\nonumber\\
&=& c_1 + \cdots + c_d + \alpha'_{d+1} c_{d+1}' + \cdots + \alpha'_{r}c_{r}'
\quad (1 \geq \alpha'_{d+1} \geq \cdots \geq \alpha'_r \geq 0)
\end{eqnarray}
where $e_{p+1}, \ldots, e_r $ are pairwise orthogonal minimal tripotents in $V_0(\xi) \cap \overline D$ and
$c_{d+1}' \ldots, c_r'$ are the ones in $V_0(c) \cap \overline{D}$.

Given that $\xi = e_1 + \cdots+e_p$ is a structural tripotent, we have $c_i \in V_2(\xi) \oplus V_0(\xi)$.
Since $P_2(c_i)(\xi) \neq 0$ and $c_i$ is a minimal tripotent, we must have $c_i \in V_2(\xi)$ for
$i=1, \ldots, d_0$,
by a remark following (\ref{osum}), and hence $c_0=c_1 + \cdots +c_{d_0}\in V_2(\xi)$.
In particular, $$c_0 \bo (e_{p+1} + \cdots +e_{r}) =0.$$
We deduce from (\ref{w11}) that
 $$c_0 \bo c_0 =(c_1 + \cdots +c_{d_0}) \bo (c_1 + \cdots +\alpha'_{r} c_{r}') = c_0 \bo w = c_0 \bo \xi$$
 and
 \begin{equation}\label{xc0}
 \xi = c_0 + (\xi-c_0) \in c_0 + V_0(c_0)
 \end{equation}
 which implies $\max \{\|c_0\|, \|\xi-c_0\|\} = \|\xi\|=1$ by orthogonality.

Consequently, $x\in V_0(\xi)\cap D$ implies $x \in V_0(c_0)$ since $V_2(\xi) \bo V_0(\xi)= \{0\}$ and
$$\{c_0,c_0, x\} = \{c_0, \xi,x\}=0.$$ Therefore $(\xi - c_0) \bo x = \xi \bo x - c_0\bo x =0$
and
$$\|( \xi- c_0) + x\| = \max\{\| \xi-c_0\|, \|x\|\} \leq 1$$ by orthogonality,
which yields
\begin{eqnarray*}\label{cperp}
&&\ell(D) \subset  \Gamma_\xi = \xi + (V_0(\xi) \cap D)\\
& = & c_0  + (\xi-c_0)+ (V_0(\xi) \cap D)
 \subset  c_0+ (V_0(c) \cap \overline D)=\overline\Gamma_{c_{_0}}.
\end{eqnarray*}
This completes the proof.
\end{proof}

\begin{remark}\label{dw1} If $D$ is a polydisc in the preceding theorem,
Lemma \ref{sum} implies a sharper result than (\ref{xc0}), namely, $\{c_1, \ldots, c_{d_0}\} \subset \{e_1, \ldots, e_p\}$.
In fact, this can be delivered by the  more direct arguments below, with the previous notation. The following derivation will be
made use of in the Appendix.

Given any $a \in H(c,s)$ for some $s>0$, the $f$-invariance of $H(c,s)$ implies $f^m(a) \in H(c,s)$ for all $m$, that is,
\begin{eqnarray*}
s &>&F_c(f^m(a)) = \|\sum_{1\leq i\leq j\leq d}\rho_i\rho_jP_{ij}B(c,f^m(a))B(f^m(a),f^m(a))^{-1/2}\|\\
&\geq & \|\sum_{1\leq i\leq j\leq d}\rho_i\rho_jP_{ij}B(c,f^m(a))B(f^m(a),f^m(a))^{-1/2}(e_{mk})\|
\quad (k =1, \ldots, r)\\
&=& \|\sum_{1\leq i\leq j\leq d}\rho_i\rho_jP_{ij}B(c,f^m(a))(1-\alpha_{mk}^2)^{-1}(e_{mk})\|
\quad \mbox{(by (\ref{eq:Bergman Neg Sq Rt}))}\\
&\geq &
 \|P_{ii}(\sum_{1\leq i'\leq j'\leq d}\rho_{i'}\rho_{j'}P_{i'j'}B(c,f^m(a))(1-\alpha_{mk}^2)^{-1}(e_{mk}))\|
 \quad  (\|P_{ii}\|\leq 1)
\end{eqnarray*}
In particular,
\begin{equation}\label{ki'}
\|P_{ii}B(c,f^m(a))(e_{mk})\| < \rho_i ^{-2} (1-\alpha_{mk}^2) s \quad (i = 1, \ldots, d;\,k =1, \ldots, p)
\end{equation}
where $P_{ii}= P_2(c_i)$ and $\lim_m \alpha_{mk} =\alpha_k =1$. Let $m\rightarrow \infty$. Then we have
\begin{equation}\label{k=1r}
\|P_2(c_i)(e_k - 2\{c,  e_k, e_k) +\{c,  e_k, c\})\|=\|P_2(c_i)B(c, \ell (a))(e_k) \|
\leq \rho_i^{-2} (1-\alpha_k^2) s =0.
\end{equation}
 Now, for $i \in \{1, \ldots, d_0\}$, we have $P_2(c_i)\xi \neq 0$ which implies $P_2(c_i)(e_k)
=\lambda_k c_i \neq 0$ for some $k\in \{1, \ldots, p\}$. As $D$ is abelian, we have $e_k= \lambda_kc_i $ with  $|\lambda_k|=1$,
by Lemma \ref{24}.
 From (\ref{k=1r}), we deduce via Lemma \ref {c(cee)c} that
\begin{eqnarray*}
&&0 = \lambda_k c_i - 2 P_2(c_i)\{c, e_k,e_k\} + P_2(c_i)\{c, e_k,c\}\\
&=& \lambda_k c_i -2|\lambda_k|^2 c_i -2\{c_i, \{c_i, e_k^1, e_k^1\}, c_i\} +\overline\lambda_k c_i\\
&=& 2{\rm Re}\,\lambda_k c_i -2|\lambda_k|^2 c_i -2\{c_i, \{c_i, e_k^1, e_k^1\}, c_i\}
\end{eqnarray*}
where $e_k^1 = P_1(c_i)(e_k)=0$ since $D$ is abelian. It follows that

\begin{equation}\label{c_i}{\rm Re}\,\lambda_k  -|\lambda_k|^2=0
\end{equation}
and $\lambda_k =1$.  Hence $\ddot{}c_i =e_k$ for some $k\in \{1, \ldots, p\}$.
\end{remark}

\begin{xexam}
Let $D$ be a Lie ball and $f: D \longrightarrow D$ a compact fixed-point free holomorphic
map. Under the condition of Theorem \ref{dw}, either the iterates $(f^n)$ converge locally uniformly
to a constant function $f_0(\cdot) = \xi \in \partial D$, or there is a minimal tripotent $c\in \partial D$
such that each limit function  $\ell= \lim_m f^m$
is of the form
$$\ell(a) = c + \lambda_{_\ell} c^* \qquad ( a\in D)$$
for some $ |\lambda_{_\ell}|=1$.

To see this, let $\Gamma_c$ be the boundary component such that $\ell(D)\subset \overline \Gamma_c$
for every  limit function $\ell= \lim_m f^m$.

Given the orbit $\mathcal{O}(a_0)$ in Theorem \ref{dw},
$\ell(a_0) \in \Sigma^*(D) = \Sigma(D)$, by Example \ref{sigma*},
which is a maximal tripotent, and $\ell(D)= \{\ell(a_0)\}$.

 If $c$ is a maximal tripotent, then  $\overline\Gamma_c =\{c\} =\{\ell(a_0)\}$. This implies that all subsequences $(f^m)$
 of $(f^n)$ converge locally unifomrly to the same limit $f_0(\cdot)=c$, and hence, so does $(f^n)$, by Lemma \ref{cp}.

 If $c$ is a minimal tripotent, then $\overline\Gamma_c=c+ \overline{\mathbb{D}}e^*$
 and hence $\ell(a)=\ell (a_0)= c+ \lambda_{_\ell} c^*$ for some $|\lambda_{_\ell}|=1$ and  for all $a\in D$.
 \end{xexam}

\begin{xexam}
Let $D = D_1 \times \cdots
\times D_p$ ($p>1$) be a product of  Hilbert balls and $f: D \longrightarrow D$
a fixed-point free compact  holomorphic map. Under the condition of Theorem \ref{dw} and
by Example \ref{bergph1},  there is a non-empty set $ J \subset \{1, \ldots, p\}$
such that all limit functions $\displaystyle \ell = \lim_{k\rightarrow \infty} f^{n_k}$ of the iterates $(f^n)$
take  values
in the closed face
$$\overline\Gamma = \overline\Gamma_1 \times \cdots \times \overline\Gamma_p $$
of $\overline D$, where
\[\overline\Gamma_j = \left\{\begin{matrix} \{\xi_j\} & (j\in J, \|\xi_j\|=1)\\\\
                                                         \overline D_j & (j\notin J).
                                 \end{matrix}\right.
\]
\end{xexam}

In the exceptional case of the bidisc  $\mathbb{D}^2$,
the condition on $\Sigma^*(\mathbb{D}^2)$ in Theorem \ref{dw} can be
dispensed with, and
 the  face $\overline\Gamma$ has one of the following forms:
$$\{(e^{i\alpha}, e^{i\beta})\}, \quad \{e^{i\alpha}\}\times \overline{ \mathbb{D}}, \quad \overline{\mathbb{D}}\times \{e^{i\beta}\}
\qquad (\alpha, \beta \in \mathbb{R}).$$
If $\ell(\mathbb{D}^2) \subset \{e^{i\alpha}\}\times \overline{\mathbb{D}}$ for all
subsequential limits $\ell$, then
$$\ell(x,y) = (e^{i\alpha}, \pi_2 \ell(x,y))\quad {\rm and}\quad  \pi_2 \ell:\mathbb{D}\times \mathbb{D}
\longrightarrow \overline{\mathbb{D}} ~\mbox{is holomorphic}$$
whereas $\ell(\mathbb{D}^2) \subset \overline{\mathbb{D}}\times \{e^{i\beta}\}$ for all
$\ell$ implies
$$\ell(x,y) = (\pi_1 \ell(x,y), e^{i\beta})$$
where $\pi_i : (x_1,x_2) \in \overline{\mathbb{D}}^2\mapsto x_i\in \overline{\mathbb{D}}$ ($i=1,2$)
are the coordinate maps.
These are the alternatives derived by Herv\'e \cite[Th\'eor\`eme 4]{h1},
using protracted arguments specific for the bidisc  $\mathbb{D}^2$.
We expose  in the Appendix the special feature of $\mathbb{D}^2$ in the arguments and provide
a simplification of Herv\'e's proof, using horofunctions and invariant horoballs.

For a compact fixed-point free holomorphic self-map $f$ on a single Hilbert ball $D$,
the result of Chu and Mellon in \cite[Theorem]{cm}
 is the simplest case of  Theorem \ref{dw} for  rank-$1$  domains.
 In this instance, the iterates $(f^n)$ converge locally uniformly to a
constant map $f_0(\cdot) = \xi\in \partial D$ by Lemma \ref{cp}. Without the compactness assumption on $f$,
this result is false in infinite dimension as noted before. Instead, we have the following alternatives.

\begin{theorem}\label{single} Let $D$ be a Hilbert ball and $f: D \longrightarrow D$ a fixed-point free holomorphic map.
Then there is a boundary point $\zeta \in \partial D$ such that either $\ell(D) \subset D$ or $\ell(D) =\{\zeta\}$
for all limit functions $\ell=\lim_k f^{n_k}$. The former inclusion cannot occur if $\dim D <\infty$.
\end{theorem}
\begin{proof} By Lemma \ref{g}, there is a sequence $(z_k)$ in $D$ converging to
some $\zeta \in \partial D$, which defines the $f$-invariant  horoballs
$H(\zeta,s)$ for $s>0$.

Let $\ell = \lim_k f^{n_k}$ be a limit function. Pick  $a\in D$. Then $\ell(a) \in \overline D$
and hence either $\ell(a) \in D$ or $\ell (a) \in \partial D$. In the former case,
we have $\ell(D) \subset D$ by Lemma \ref{3.2}. In the latter case, we have $a\in H(\zeta,s)$ for some $s>0$,
which implies, by $f$-invariance,
$$\ell(a) \in \overline H(\zeta,s) \cap \partial D = \{\zeta\}$$
 by Theorem \ref{1ptintersection}.
Therefore Lemma \ref{3.2} implies $\ell(D) \subset \Gamma_{\ell(a)} =\{\ell(a)\}= \{\zeta\}$ as $\ell(a)$
is a maximal tripotent.

If $\dim D<\infty$, the map $f$ is compact and hence $\ell(D) \subset \partial D$, by Lemma \ref{kry}.
\end{proof}

\begin{xrem} The preceding proof simplifies considerably the one given by Chu and Mellon in \cite[Theorem]{cm}
for {\it compact} holomorphic maps.
\end{xrem}

\section{Appendix}

In the exceptional  case of the bidisc $D=\mathbb{D}^2$, Theorem \ref{dw},
 without the condition on $\Sigma^*(\mathbb{D}^2)$,
has been proved by Herv\'e
\cite{h1} with protracted arguments specific for the bidisc  $\mathbb{D}^2$, via
the Cayley transform of $\mathbb{D}$ onto the right-half complex plane and Harnack's theorem
for harmonic functions.
We simplify Herv\'e's proof (still dependent on $\mathbb{D}^2$) below, using horofunctions and invariant horoballs.

Let
$$\pi_1: (x,y) \in \overline{\mathbb{D}}^2 \mapsto x\in \overline{\mathbb{D}}, \quad \pi_2: (x,y) \in
\overline{\mathbb{D}}^2 \mapsto y\in \overline{\mathbb{D}} $$
be the coordinate maps.

Given a fixed-point free
holomorphic self-map $f$ on $\mathbb{D}^2$, Herv\'e first observed in \cite[$7^0$]{h1} that, by a simple reduction,
one need only consider the following three mutually exclusive cases (in our notation):
\begin{enumerate}
\item[(a)] There is a boundary point $\zeta=(e^{i\alpha},0)$ such that for each
$(x,y)\in \mathbb{D}^2$, the image $f(x,y)$ is contained in a horoball
$$ H(\zeta, s_x)=\mathbb{H}(e^{i\alpha},s_x)\times \mathbb{D}\qquad (s_x >0)$$
 with $x\in \partial \mathbb{H}(e^{i\alpha},s_x)$,
 and there is holomorphic function $\eta: \mathbb{D}\longrightarrow \mathbb{D}$
such that $\pi_2f(x,y)=y$ if and only if $y=\eta(x)$.
\item[(b)] There is a boundary point $\zeta=(e^{i\alpha},e^{i\beta})$ such that for each
$(x,y)\in \mathbb{D}^2$, the image $ f(x,y)$ is contained in 
$$\mathbb{H}(e^{i\alpha},s_x) \times \mathbb{H}(e^{i\beta},s_y) \qquad (s_x, s_y >0)$$
with $(x,y) \in \partial \mathbb{H}(e^{i\alpha},s_x)\times \partial \mathbb{H}(e^{i\beta},s_y)$.
\item[(c)] There are two holomorphic functions $\nu, \eta: \mathbb{D}\longrightarrow \mathbb{D}$ such that
$$\pi_1f(x,y) = x \Leftrightarrow x= \nu(y), \quad {\rm and} \quad  \pi_2f(x,y) = y \Leftrightarrow y= \eta(x).$$
\end{enumerate}

In case (b),   one can follow the arguments in \cite[$12^0$]{h1} to establish that,
given  $(x,y) \in \mathbb{D}^2$, either
$$\lim_{n\rightarrow \infty} \pi_1 f^n(x,y) =e^{i\alpha}\quad {\rm or}
\quad  \lim_{n\rightarrow \infty} \pi_2 f^n(x,y) = e^{i\beta}.$$
 It follows that, in the former case,
for each limit function $\ell= \lim_m f^{m}$ and $a\in \mathbb{D}^2$,
we have $$\ell(a) = e_1 + \alpha e_2 \quad {\rm with} \quad 0\leq \alpha \leq 1, \quad e_1 = (e^{i\alpha},0), e_2=(0, e^{i\theta})$$
 and $\ell(\mathbb{D}^2) \subset \overline \Gamma_{(e^{i\alpha},0)}$.

 Likewise, in the latter case, we have
$\ell(\mathbb{D}^2) \subset \overline \Gamma_{(0, e^{i \beta})}$
 for all limit functions $\ell$.

We deal with cases (a) and (c) in the following proof.

\begin{xthm}\label{dw2}{\it
Let $f: \mathbb{D}^2 \longrightarrow \mathbb{D}^2$ be a fixed-point free holomorphic map.
Then there is a holomorphic boundary component $\Gamma$ in the boundary $\partial \mathbb{D}^2$ such that
$\ell(\mathbb{D}^2) \subset \overline \Gamma$ for all limit functions $\ell= \lim_k f^{n_k}$ of the iterates
$(f^n)$.}
\end{xthm}

\begin{proof}
It suffices to consider cases (a) and (c) above. For this, one need only consider
 the $f$-invariant horoballs in $\mathbb{D}^2$
 of the form
\begin{equation}\label{H1}
H(c,s)  =\mathbb{H}(e^{i\alpha},s)\times \mathbb{D}
\end{equation}
or of the form
\begin{equation}\label{H2}
H(c,s) = \mathbb{H}(e^{i\alpha},s) \times \mathbb{H}(e^{i\beta},s/\sigma) \qquad (0<\sigma\leq 1)
\end{equation}
which are defined by a sequence $(z_k)$  converging to $c$ via Lemma \ref{g}, where we assume
$c=(e^{i\alpha},0)$ in (\ref{H1}) without loss of generality.

Indeed, if $\lim_k z_k = (0, e^{i\beta})$
so that $H(c,s) = \mathbb{D} \times \mathbb{H}(e^{i\beta},s)$,
 we can make use of the transformation
$T: (x,y) \in \mathbb{D}^2 \mapsto (y,x)\in \mathbb{D}^2$ and consider the transformed sequence $(Tz_k)$
which converges to $(e^{i\beta},0)$, with the holomorphic map $\widetilde f = TfT$ and the iterates $\widetilde f^n = Tf^nT$.

Let $0<s<1$ and $a\in H(c,s)$ in (\ref{H1}) or (\ref{H2}).
In cases (a) and (c), the Lemma in \cite[p.\,11]{h1} reveals a dichotomy in that there exist $\alpha_1,\alpha_2>0$, depending on $a$,
such that for each $n\in \mathbb{N}$,
the following inequalities cannot hold simultaneously:
\begin{equation}\label{dicot}
\frac{1-|\pi_2f^n(a)|}{1-|\pi_1 f^n(a)|}<\alpha_1 \quad {\rm and} \quad \frac{1-|\pi_2 f^{n+1}(a)|}{1-|\pi_2 f^n(a)|}< 1+ \alpha_2.
\end{equation}

Making use of this, we first  show  that either
\begin{equation}\label{extreme0}
\lim_n| \pi_1 f^n(a)| =1
\end{equation}
 or there is a convergent subsequence $(f^m)$ of the iterates $(f^n)$ such that
\begin{equation}\label{extreme}
\lim_m |\pi_1 f^m(a)|= \lim_m |\pi_2 f^m(a)|=1.
\end{equation}

Observe that $\|f^n(a)\|= \max (|\pi_1 f^n(a)|, |\pi_2 f^n(a)|)$ and $\lim_n \|f^n(a)\|=1$ by Lemma \ref{kry}.
If from some $n$ onwards, $|\pi_1 f^n(a)| > |\pi_2 f^n(a)|$, then we have
$$\lim_n |\pi_1 f^n(a)| = \lim_n \|f^n(a)\|=1$$
and we are done. Otherwise, there is a subsequence $(f^m)$ of $(f^n)$ such that $|\pi_1 f^m(a)| \leq |\pi_2 f^m(a)|$
and we may assume that $(f^m)$ is convergent by choosing a subsequence if necessary. Hence
$$\lim_m |\pi_2 f^m(a)| = \lim_m \|f^m(a)\|=1.$$
Let $h = \lim_m f^m$.
If we also have $\lim_m |\pi_1f^m(a)|=1$, then we are done. Otherwise,
\begin{equation}\label{extreme1}
\pi_1 h (a) = \lim_m\pi_1 f^m(a) \in \mathbb{D}
\end{equation}
and since $\|h(a)\|=1$, we would have a spectral decomposition
$$h (a) = u_1 + \gamma u_2 \in \Gamma_{u_1} \quad \mbox{where $u_1 = (0, e^{i\theta_1})$,
$\theta_1 \in \mathbb{R}$ and $0\leq \gamma <1$}.$$
It follows that $h(\mathbb{D}^2) \subset \Gamma_{u_1}$. In particular, $h(f(a)) \in  \Gamma_{u_1}$ implies
$\lim_m |\pi_1f^{m+1}(a)| = |\pi_1 h(f(a))| <1$. Hence $\|h(f(a))\|=1$ implies $\lim_m|\pi_2 f^{m+1}(a) | = |\pi_2(h(f(a))| =1$
 and from some $m$ onwards, we have $|\pi_1 f^{m+1}(a)| < |\pi_2f^{m+1}(a)|$ as well as
 $|\pi_1 f^{m}(a)| < |\pi_2f^{m}(a)|$.

Since $\lim_n \|f^n(a)\|=1$, the sequence $(\|f^n(a)\|)$ is increasing from some $n$ onwards. Therefore, from some
$m$ onwards, we have
$$|\pi_2 f^{m+1}(a)| =\|f^{m+1}(a)\| \geq \|f^{m}(a)\| =|\pi_2 f^m(a)|.$$

Now Herv\'e's dichotomy gives the contradiction that, from some $m$ onwards,
$$\frac{1-|\pi_2f^m(a)|}{1-|\pi_1 f^m(a)|}<\alpha_1\quad {\rm and} \quad
1\geq \frac{1-|\pi_2f^{m+1}(a)|}{1-|\pi_2 f^m(a)|}\geq 1+\alpha_2 >1$$
where $\lim_m \frac{1-|\pi_2f^m(a)|}{1-|\pi_1 f^m(a)|}=0$.
Hence (\ref{extreme1}) cannot occur and  the claim is proved.

Now let $\ell=\lim_k f^{n_k}$ be a limit function with a spectral decomposition
$$\ell(a) = \lim_k f^{n_k}(a) = e'_1 + \alpha_2 e'_2 \in \overline{\mathbb{D}}^2 \qquad (0\leq \alpha_2 \leq 1).$$
Then $f^{n_k}(a) \in H(c, s)$ implies
$$\left\|\sum_{1\leq i'\leq j'\leq d_c}\rho_{i'}\rho_{j'}P_{i'j'}B(c,f^{n_k}a))B(f^{n_k}(a),f^{n_k}(a))^{-1/2}\right\|
=F_c(f^{n_k}(a) ) <s.$$
where $d_c =1$ for $H(c,s)$ in ( \ref{H1}), and $d_c =2$ in  (\ref{H2}).

Write $c_1=(e^{i\alpha},0)$ and $c_2 = (0, e^{i\beta})$. In the special case of the bidisc $\mathbb{D}^2$, we must have
either $c_1 \bo e_1' \neq 0$ or $c_2 \bo e_1' \neq 0$, which implies either $c_1 = e^{i\theta_1} e_1'$ or
$c_2 = e^{i\theta_2}e_1'$ for some $\theta_1, \theta_2 \in \mathbb{R}$, by Lemma \ref{24}.

As in the the derivation of (\ref{ki'}) and (\ref{c_i}) in Remark \ref{dw1}, we have, for $c_j= e^{i\theta_j} e'_1$
(j=1, 2),
\begin{eqnarray}\label{final}
\| P_{2}(c_j) B(c,\,\ell(a))(e'_1)\| &=& \lim_{k\rightarrow \infty} \,\| P_{2}(c_j) B(c,\,f^{n_k}(a))(e'_{n_k,1})\|\nonumber\\
&\leq& \lim_{k\rightarrow \infty} \rho_j^{-2} (1-\alpha_{n_k, 1}^2)s =0
\end{eqnarray}
where $f^{n_k}(a) =\alpha_{n_k,1} e'_{n_k,1} + \alpha_{n_k, 2} e'_{n_k,2} \rightarrow e'_1 + \alpha_2 e'_2$ as $k \rightarrow \infty$.
This gives
 $$({\rm Re}\, e^{i\theta_j} -|e^{i\theta_j}|^2)c_j =\{c_j, \{c_j, (e'_1)^1, (e'_1)^1\}, c_j\}=0$$
where $(e'_1)^1=P_1(c_j)(e'_1)=0$ because $\mathbb{D}^2$ is abelian. Hence $e^{i\theta_j}=1$ and $c_j= e'_1$.

It follows that, for the case  $\lim_n |\pi_1 f^n(a)|=1$ in (\ref{extreme0}), we have $c_1 \bo e'_1 \neq 0$ which
gives $c_1 =e'_1$ and
$$\ell(\mathbb{D}^2) \subset \overline \Gamma_{(e^{i\alpha},0)}$$
for each limit function $\ell= \lim_k f^{n_k}$.

Finally, consider the case (\ref{extreme}) in which $\lim_m |\pi_1 f^m(a)|=\lim_m |\pi_2 f^m(a)|=1$. We can write
\begin{eqnarray*}
h(a) &=& (e^{i\theta_1}, e^{i\theta_2})= e_1 +  e_2 \quad{\rm with} \quad  e_1 =(e^{i\theta_1},0), ~ e_2=(0, e^{i\theta_2}) \\
&=& \lim_m f^m(a) = \lim_m (\alpha_{m1}e_{m1} +\alpha_{m2}e_{m2}).
\end{eqnarray*}
Since $e_j= \lim_m e_{mj}$ for $j=1,2$ and $\mathbb{D}^2$ is abelian, we have $e_{mj}= e^{i\theta_{mj}}e_j$
with $\theta_{mj}\in \mathbb{R}$ from some $m$ onwards, by Lemma \ref{24},
and $\lim_m \alpha_{mj}e^{\theta_{mj}}=1$.

By Corollary \ref{anotherseq}, the horofunction $F_c$ for the horoballs in (\ref{H1}) can be computed
via a sequence $(\beta_{k1}c_1)$,
where $c_1 = (e^{i\alpha},0)$ and $\lim_{k\rightarrow\infty} \beta_{k1} =1$, whereas the horofucntion $F_c$
for (\ref{H2}) can be computed by a sequence $(\beta_{k1}c_1 + \beta_{k2}c_2)$ with
$c_1= (e^{i\alpha},0)$, $c_2=(0, e^{i\beta})$ and $\lim_k \beta_{kj}=1$ for $j=1,2$.
We show both horofunctions satisfy
\begin{equation}\label{zero}
\lim_{m\rightarrow \infty} F_c(f^m(a) ) =0.
\end{equation}
To show this, we first observe that, by
making use of (\ref{final}) as before, we have $e_1 =c_1= (e^{i\alpha},0)$ in (\ref{H1})
whereas in case (\ref{H2}), we have $e_1 = c_1$ and $e_2 =c_2= (0, e^{i\beta})$.
Hence the horofunction in  case (\ref{H1}) is given by
\begin{eqnarray*}
&&F_c(f^m(a)) = \lim_{k\rightarrow \infty}\frac{1-\beta_{k1}^2}{1-\|g_{-\beta_{k1}c_1}(\alpha_{m_1}e^{i\theta_{m1}}e_1
+\alpha_{m2}e^{i\theta_{m2}}e_2)\|^2}\\
&=&\lim_{k\rightarrow \infty}\frac{1-\beta_{k1}^2}{1-\|\psi_{-\beta_{k1}}(\alpha_{m_1}e^{i\theta_{m1}})c_1
+\alpha_{m2}e^{i\theta_{m2}}c_2\|^2} \quad ({\rm Lemma}~ \ref{k=0})\\
&=& \lim_{k\rightarrow \infty}\frac{1-\beta_{k1}^2}{1-|\psi_{-\beta_{k1}}(\alpha_{m_1}e^{i\theta_{m1}})|^2}
=  \frac{|1-\alpha_{m_1}e^{i\theta_{m1}}|^2}{1-|\alpha_{m_1}e^{i\theta_{m1}}|^2} \quad (\rm{by}~\ref{hh})
\end{eqnarray*}
where $\lim_{k\rightarrow \infty} |\psi_{-\beta_{k1}}(\alpha_{m_1}e^{i\theta_{m1}})| =1 > |\alpha_{m2}e^{i\theta_{m2}}|$.

In case  (\ref{H2}), we have
\begin{eqnarray*}
F_c(f^m(a))& =& \lim_{k\rightarrow \infty}\frac{1-\beta_{k1}^2}{1-\|g_{-(\beta_{k1}c_1+\beta_{k2}c_2)}(\alpha_{m1}e^{i\theta_{m1}}e_1
+\alpha_{m2}e^{i\theta_{m2}}e_2)\|^2}\\
 &=&\lim_{k\rightarrow \infty}\frac{1-\beta_{k1}^2}{1-\|\psi_{-\beta_{k1}}(\alpha_{m1}e^{i\theta_{m1}})c_1
+\psi_{-\beta_{k2}}(\alpha_{m2}e^{i\theta_{m2}})c_2\|^2}\\
 &=&\lim_{k\rightarrow \infty}\frac{1-\beta_{k1}^2}{1-\max(|\psi_{-\beta_{k1}}(\alpha_{m1}e^{i\theta_{m1}})|^2,
|\psi_{-\beta_{k2}}(\alpha_{m2}e^{i\theta_{m2}})|^2)}\\
&= &\max(\frac{|1-\alpha_{m1}e^{i\theta_{m1}}|^2}{1-|\alpha_{m1}e^{i\theta_{m1}}|^2}, \frac{|1-\alpha_{m2}e^{i\theta_{m2}}|^2}{1-|\alpha_{m2}e^{i\theta_{m2}}|^2}).
\end{eqnarray*}
It follows that $\lim_{m\rightarrow \infty} F_c(f^m(a) ) =0$ in both cases since for $z=re^{i\theta}\in \mathbb{D}$,
$$\frac{|1-z|^2}{1-|z|^2}
= \frac{1-2r\cos \theta +r^2}{1-r^2} = \frac{1}{{\rm Re}\,\left(\frac{1+z}{1-z}\right)} \quad {\rm and} \quad
\lim_{z\rightarrow 1}\frac{|1-z|^2}{1-|z|^2}=0.
$$

We conclude by showing that, in each case of (\ref{H1}) and (\ref{H2}), we have
$$\ell(\mathbb{D}^2) \subset \overline\Gamma_c$$
for every limit function $\ell= \lim_k f^{n_k}$.

Indeed, for each $s>0$,  there exists $m_s\in \mathbb{N}$, from (\ref{zero}), such that $f^m(a) \in H(c,s)$ for $m \geq m_s$.
Hence for all $n_k > m_s$, the $f$-invariance of $H(c,s)$ implies
$$f^{n_k}(a) = f^{n_k-m_s}(f^{m_s}(a)) \in H(c,s)$$
 and therefore $\ell(a) = \lim_k f^{n_k}(a) \in \overline H(c,s)$.
As $s>0$ was arbitrary, we have shown $\ell(a) \in \bigcap_{s>0} \overline H(c,s) = \overline \Gamma_c$
and so $\ell(\mathbb{D}^2) \subset \overline\Gamma_c$.
\end{proof}

\end{document}